\documentclass[12pt,reqno]{amsart}
\input{mypreamble}
\usepackage[utf8]{inputenc}

\usepackage{graphicx} % Required for inserting images

\title{Compatibility of $F$-isocrystals on adjoint Shimura varieties}
\author{Jake Huryn, Kiran S. Kedlaya, Christian Klevdal, and Stefan Patrikis}
\address{Department of Mathematics, The Ohio State University, 100 Math Tower, 231 West 18th
Avenue, Columbus, OH 43210, USA}
\email{huryn.5@osu.edu}
\address{Department of Mathematics, University of California San Diego, 9500 Gilman Drive, La Jolla, CA 92093, USA}
\email{kedlaya@ucsd.edu}
\email{christian.klevdal@gmail.com}
\address{Department of Mathematics, The Ohio State University, 100 Math Tower, 231 West 18th
Avenue, Columbus, OH 43210, USA}
\email{patrikis.1@osu.edu}
\date{}

\begin{document}

\begin{abstract}
In this article, we extend the main results of \cite{patrikis-klevdal:shimura} to include compatibility of canonical $\ell$-adic local systems and canonical $F$-isocrystals on Shimura varieties in the superrigid regime. Our method relies on the crystallinity of canonical $p$-adic local systems of Esnault--Groechenig \cite{pila-shankar-tsimerman:andre-oort}, \cite{esnault-groechenig:cristallinity} as well as Margulis superrigidity and the crystalline-to-\'etale companion construction of Drinfeld \cite{drinfeld:pross}, Abe-Esnault \cite{abe-esnault:lefschetz}, and Kedlaya \cite{kedlaya:etale-crystalline-1}. 
\end{abstract}
\maketitle

%\tableofcontents

\section{Introduction}
In his 1979 article in the Corvallis proceedings \cite{deligne:canonicalmodels}, Deligne constructs from a Shimura datum $(G,X)$ a tower of Shimura varieties $\{\Sh_K(G,X)\}_K$ indexed by compact open subgroups $K$ of $G(\mbb{A}_f)$.
A guiding principle in this theory is the expectation that Shimura varieties should be moduli spaces of motives with $G$-structure.\footnote{
This is not technically correct for general $G$, where Shimura varieties are expected to parametrize motives with $G^c$-structure for $G^c$ a certain central quotient of $G$.
In the body of the work, we assume $Z_G(\QQ)$ is discrete in $Z_G(\mbb{A}_f)$, which ensures that $G = G^c$, so we make this assumption throughout the introduction as well.
}
The largest class of Shimura varieties which can be related (in a precise sense) to moduli spaces of motives are the abelian type Shimura varieties; the motives in this case arise from abelian varieties with certain Hodge cycles.
This still leaves a large class of Shimura varieties, including all exceptional Shimura varieties, where there is no known moduli-theoretic interpretation.
Nevertheless, the `motivic expectation' often gives concrete predictions for important properties of Shimura varieties, which can sometimes be proven unconditionally.
One such example is the theory of canonical models (proven for abelian type Shimura varieties by Deligne \cite{deligne:canonicalmodels} and in general by Milne \cite{milne:canonical-models}), which shows that the Shimura varieties $\Sh_K(G,X)$, \emph{a priori} defined over $\CC$, are in fact naturally defined over a number field $E(G,X) \subseteq \CC$, the reflex field.
A much deeper prediction is that  $\Sh_K(G,X)$ should have an \emph{integral} canonical model, defined over the ring of integers of $E(G,X)$ (at least away from finitely many primes).
The conjecture of Langlands--Rapoport gives a precise description of the mod $p$-points of the integral canonical models of Shimura varieties. 

In a seemingly separate direction, each Shimura variety admits canonical $\ell$-adic/$p$-adic `coefficient objects' (i.e.\ $\ell$-adic local systems/$F$-isocrystals on mod-$p$ fibers) attached to each representation of $G$.
These are expected to be the $\ell$-adic \'etale and crystalline realizations of a universal family of motives, and thus one predicts that these coefficient objects form a compatible system, i.e.\ that the semisimple Frobenius conjugacy classes at closed points on an integral canonical model are rational and independent of $\ell$.
When $(G,X)$ is abelian type, this and much more is known by work of Kisin \cite{kisin:integral-models-abelian-type}, \cite{kisin:mod-p}. 

The main focus of this article is the question of compatibility of the canonical coefficient objects in the \emph{non}-abelian type setting, where the essential case to consider is when $G^\ad$ is $\QQ$-simple and has real rank $\geq 2$.
We thus make this assumption.
In a previous article \cite{patrikis-klevdal:shimura}, two of the authors prove the compatibility of the canonical $\ell$-adic local systems attached to representations of $G^\ad$, at least for $\ell\neq p$. The main result of this article completes the picture and proves compatibility of the $\ell$-adic local systems and the canonical $F$-isocrystals. 

In order to state the result, we introduce some notation.
Fix a neat compact open subgroup $K_0 \subseteq G(\mbb{A}_f)$ and a point $s \in \Sh_K(G,X)(\CC)$. Let $S$ be the geometrically connected component of $\Sh_{K_0}(G,X)$ containing $s$; it is defined over a finite extension $E$ of the reflex field of $(G,X)$.
For each prime $\ell$, we write
    \[ \rho_{\ell} \colon \pi_1(S,s) \to G(\QQ_\ell), \qquad \rho_{\ell}^\ad \colon \pi_1(S,s) \to G^\ad(\QQ_\ell) \]
for the canonical $G(\QQ_\ell)$ and $G^\ad(\QQ_\ell)$-local systems on $S$ (see e.g.\ \cite[\S3.1]{patrikis-klevdal:shimura} or \cite[\S4]{cadoret-kret}).
As explained in  \cite{patrikis-klevdal:shimura}, there exists an integer $N$ and integral model $\ms{S}$ of $S$ that is flat over $\mc{O}_E[1/N]$ and such that for each prime $\ell$, $\rho_\ell^\ad$ extends
%\footnote{In fact, \cite{patrikis:full-compatibility} shows that we can find $N, \ms{S}$ such that for each $\ell$, the canonical $G(\QQ_\ell)$-local systems extend to $\ms{S}[1/\ell]$.}
(necessarily uniquely) to a $G^\ad(\QQ_\ell)$-local system
        \[ \rho_\ell^\ad \colon \pi_1(\ms{S}[1/\ell], s) \to G^\ad(\QQ_\ell).\]
Let $v \in \Spec(\mc{O}_E[1/N])$ be a closed point of residue characteristic $p$, and let $\ms{S}_v$ be the fiber of $\ms{S}$ over $v$.
As a consequence of the main result of \cite{patrikis-klevdal:shimura}, the following holds after enlarging $N$:
For any representation $\xi$ of $G^\ad$, any closed point $x \in \ms{S}_v$, and any $\ell \neq p$, the characteristic polynomial of $(\xi\circ \rho_{\ell}^\ad)(\Frob_x)$ has coefficients in $\QQ$ and is independent of $\ell$.

With this in mind, we now state a simplified version of the main theorem of the present paper (\cref{thm:main}).
See Definition \ref{defn:G-F-isocrystal} and Example \ref{eg:isocrystals-over-point} for the terminology used here.
%Also, the statements require a further enlargement of $N$ (see Remark \ref{rmk:loss-of-primes} for details).

\begin{thm}\label{thm:main-intro}
With notation as above and after further enlarging $N$,\footnote{
See Remark \ref{rmk:loss-of-primes} for details.
}
there exists an overconvergent $G$-$F$-isocrystal $\mc{E}_v^\dagger$ on $\ms{S}_v$ that is compatible with the canonical $G^\ad(\QQ_\ell)$-local systems in the following sense: For any representation $\xi$ of $G^\ad$, any closed point $x \in \ms{S}_v$, and any $\ell \neq p$, the characteristic polynomial of the linearized Frobenius on $x^\ast \mc{E}_v^\dagger(\xi)$ has coefficients in $\QQ$ and is equal to the characteristic polynomial of $(\xi \circ \rho_{\ell}^\ad)(\Frob_x)$. 
\end{thm}

\begin{rmk}\label{rmk:full-compatibility}
Work of the fourth author \cite{patrikis:full-compatibility} shows that the canonical $G(\QQ_\ell)$-local systems extend to a suitable integral model and uses \cref{thm:main-intro} and the main theorem of \cite{patrikis-klevdal:shimura} to establish the improved compatibility of the canonical $G(\QQ_\ell)$-local systems and $G$-$F$-isocrystals on mod-$p$ fibers, rather than just their adjoint quotients. 
\end{rmk}

\begin{rmk}\label{rmk:crystallinity-intro}

The overconvergent $G$-$F$-isocrystal $\mc{E}_v^\dagger$ in \cref{thm:main-intro} is the restriction of a logarithmic $G$-$F$-isocrystal on a compactification of $\ms{S}_v$, and is constructed using work of Esnault-Groechenig \cite{esnault-groechenig:cristallinity}. It satisfies the expected relation that the underlying convergent $G$-$F$-isocrystal is $\mbb{D}_\crys(\rho_{p,v})$, though this is only known \emph{a posteriori} from the construction (see \cref{s.main-thm-proof} for details). 
\end{rmk}

In view of the remark, Theorem \ref{thm:main-intro} has the following concrete consequence for the Galois representations arising from Shimura varieties:

\begin{cor}
With notation as above and $N$ as in Theorem \ref{thm:main-intro}, let $y\in\ms S(\cO_F[1/N_y])$, where $F/E$ is a finite extension, and let $\rho_{\ell,y}^\ad\colon\Gal_F\to G(\QQ_\ell)$ be the Galois representation obtained by restricting $\rho_\ell^\ad$ (it is unramified outside $N_y$).
Then for any representation $\xi$ of $G^\ad$, any $p\nmid N_y$ and place $v$ of $F$ above $p$, and any $\ell\neq p$, the characteristic polynomial of the linearized Frobenius on $D_\mr{crys}(\xi\circ\rho_{p,y}^\ad|_{\Gal_{F_v}})$ has coefficients in $\QQ$ and is equal to the characteristic polynomial of $(\xi\circ\rho_{\ell,y}^\ad)(\Frob_v)$.
\end{cor}

\begin{rmk}
\cref{thm:main-intro} can be reformulated in the following manner. For each closed point $x \in \ms{S}_v$ with residue field $\kappa(x)$, we define $\gamma_{\ell,x} \coloneqq \rho_{\ell}^\ad(\Frob_x) \in G(\QQ_\ell)$, and let $[\gamma_{\ell, x}^\ad]$ be the image of $\gamma_{\ell, x}$ in  $(G^\ad\git G^\ad)(\QQ_\ell)$ (here $G^\ad\git G^\ad$ is the GIT quotient of $G^\ad$ by the conjugation action, the geometric points of which correspond to semisimple conjugacy classes in $G^\ad$). We let $\gamma_{p,x} \in G(W(\kappa(x))[1/p])$ be the linearized Frobenius of $\mc{E}_v^\dagger$ at $x$, and let $[\gamma_{p,x}^\ad]$ be its image in $(G^\ad\git G^\ad)(W(\kappa(x))[1/p])$. The results of \cite{patrikis-klevdal:shimura} and this paper can be summarized by saying that there exists an element $[\gamma_{0,x}^\ad] \in (G^\ad\git G^\ad)(\QQ)$ such that $[\gamma_{0,x}^\ad] = [\gamma_{\ell,x}^\ad]$ for all primes $\ell$, including $\ell = p$.

In this framework, \cref{thm:main-intro} (combined with \cite{patrikis-klevdal:shimura} and its improvement \cite{patrikis:full-compatibility}) can be seen as a step towards associating a Kottwitz triple to mod-$p$ points of Shimura varieties, a necessary step towards formulating the Langlands--Rapoport conjecture. We refer the reader to the end of the introduction of \cite{patrikis:full-compatibility} for a more detailed discussion of Kottwitz triples and what remains to be done for non-abelian type Shimura varieties. 
\end{rmk}

\subsection{Strategy of proof and overview}
The starting point of the proof of \cref{thm:main-intro} is work of Esnault--Groechenig \cite{esnault-groechenig:cristallinity}, \cite[Appendix A]{pila-shankar-tsimerman:andre-oort} which, for a fixed faithful representation $\xi_0$ of $G$, gives an $F$-isocrystal $\mc{E}_v^\dagger(\xi_0)$ (in fact they produce something stronger, a Fontaine--Laffaille module) satisfying a crystalline comparison with $\xi_0 \circ \rho_{p,v}$. We then define $\mc{E}_v^\dagger(\xi)$ for an arbitrary representation (functorially and compatibly with tensor products) by using the canonical $G$-bundle on the Shimura variety. Once $\mc{E}_v^\dagger$ is defined, we use the crystalline-to-\'etale companions construction of \cite{drinfeld:pross}, \cite{kedlaya:etale-crystalline-1, kedlaya:etale-crystalline-2}, \cite{abe-esnault:lefschetz} and compatibility at special points to make an argument as in \cite{patrikis-klevdal:shimura}. 

In \cref{s.background-isocrystals} we give relevant background on isocrystals and Fontaine--Laffaille modules. In \cref{s.companions} we discuss companions of overconvergent $G$-$F$-isocrystals when $G$ is semisimple. In \cref{s.main-thm-statement} we recall the Shimura variety setup that we use and state the main theorems precisely. Finally, in \cref{s.main-thm-proof} we prove the main theorem. 

\subsection{Acknowledgements}
All authors were present and worked on this project at the conference ``Local Systems in Algebraic Geometry'' at The Ohio State University in May 2024, which was funded by NSF Grant No.\ DMS-2231565.
J.H.\ was supported by the NSF under
Grant No.\ DMS-2231565. S.P.\ was supported by the NSF under Grant No.\ DMS-2120325. 
K.S.K.\ was supported by the NSF under Grants No.\ DMS-2053473 and DMS-2401536, and during part of this work by the IAS (School of Mathematics) and the Simons Foundation (via the Simons Fellowship program).
We thank Marco D'Addezio for pointing out an inaccurate statement regarding isocrystals. C.K.\ would like to thank Keerthi Madapusi, Ananth Shankar and Alex Youcis for helpful conversations which led to a simplification and improved clarity in the proof of \cref{thm:overconvergent-Dcrys}-(1).  

\section{Background on isocrystals and Fontaine--Laffaille modules}\label{s.background-isocrystals}
In this section we recall conventions and definitions for isocrystals and Fontaine--Laffaille modules. 

\subsection{Modules with integrable connections}

Let $T$ be a regular integral Noetherian scheme and $X$ a smooth scheme over $T$. We denote by $\MIC(X)$ the category of modules with integrable connection relative to $T$ (for clarity we omit the $T$ from the notation; in applications, the relevant base will be clear from context).
Its objects are coherent sheaves $\mc{E}$ on $X$ with a connection $\nabla \colon \mc{E} \to \mc{E} \otimes_{\mc{O}_X} \Omega^1_{X/T}$ satisfying the Leibniz rule and $\nabla^2 = 0$. This is a rigid $\mc{O}_T(T)$-linear abelian tensor category that has an internal Hom. For a thorough treatment, see Katz \cite[\S 1]{katz:monodromy-theorem}. 
%We let $\MIC^\lf_T(X)$ be the full subcategory of $\MIC_T(X)$ consisting of objects whose underlying coherent sheaf is locally free. Then $\MIC_T^\lf(X)$ is a rigid $\mc{O}_T(T)$-linear tensor category. 

Let $j \colon X \hookrightarrow \bar{X}$ be a compactification relative to $T$ such that $\bar{X} \to T$ is smooth and proper and $D = \bar{X} \setminus X$ is a relative strict normal crossings divisor over $T$. This latter condition means that $D$ is a reduced effective Cartier divisor such that 1) there exists an \'etale map $U \to \mathbb{A}^d_T = \Spec(\mc{O}_T[t_1,\ldots, t_d])$ such that $D|_U$ is the inverse image of the vanishing locus $V(t_1\cdots t_s) \subseteq \mathbb{A}^d_T$ for some $s \leq d$; and 2) if $D = \sum_{i \in I} D_i$ is a sum of irreducible divisors $D_i$, then for any $J \subset I$, the intersection $\bigcap_{j \in J} D_j$ is smooth over $T$. 

We denote by $\MIC(\bar{X}, D)$ the category of modules with integrable connection (relative to $T$) that are logarithmic with respect to $D$. Its objects consist of coherent sheaves $\bar{\mc{E}}$ on $\bar{X}$ with a connection $\nabla \colon \bar{\mc{E}} \to \bar{\mc{E}} \otimes_{\mc{O}_{\bar{X}}} \Omega^1_{\bar{X}/T}(\log D)$ satisfying the Leibniz rule and $\nabla^2 = 0$. For each irreducible component $D_i$ of $D$ there is a residue map $\mr{Res}_{D_i} \colon \mc{T}_{\bar{X}/T}(\log D) \to \mr{End}(\bar{\mc{E}}|_{D_i})$. We let $\MIC^\nilp(\bar{X}, D)$ be the full subcategory of $\MIC(\bar{X}, D)$ consisting of objects $(\bar{\mc E},\nabla)$ such that $\bar{\mc{E}}$ is locally free and ${\nabla}$ has nilpotent residues along $D$. This is an exact rigid tensor subcategory that is closed under extensions. 

Finally, suppose $X$ is a smooth rigid analytic variety over a field $K$ that is complete with respect to a non-archimedean valuation. We let $\MIC(X)$ denote the category of modules with integrable connection, which has the same definition as for schemes. If $(\mc{E}, \nabla) \in \MIC(X)$ and $U \subseteq X$ is an open affinoid open subset such that $\Omega^1_{U/K}$ is freely generated by $dt_1,\dots,dt_d$ for some $t_1,\dots,t_d\in \mc{O}(U)$, then $\mc{O}(U)$ is a Banach $K$-algebra, and the finite $\mc{O}(U)$-module $\mc{E}(U)$ is complete with respect to an $\mc{O}(U)$-module norm which is unique up to equivalence.
We say the connection $\nabla$ is \textit{convergent} if for all such $U$ and $t_1,\dots,t_d$, the power series
\[
\sum_{i_1,\dots,i_d\geq0}
\frac{\tau_1^{i_1}\cdots\tau_d^{i_d}}{i_1!\cdots i_d!}
\cdot
\frac{\partial^{i_1}}{\partial t_1^{i_1}}
\cdots
\frac{\partial^{i_d}}{\partial t_d^{i_d}}
(v)
\]
converges for any $v\in\mc{E}(U)$ and any $\tau_1,\dots,\tau_d\in K'$ (for $K'/K$ a finite extension) with $|\tau_j|<1$.

\subsection{Isocrystals}\label{ss.(over)convergent-isocrystals}

For this subsection, we fix the following notation:

\begin{itemize}
\item
$\kappa$ is a perfect field of positive characteristic $p$ with ring of $p$-typical Witt vectors $\mc{O}_K \coloneqq W(\kappa)$, and $K \coloneqq \mc{O}_K[1/p]$ is the field of fractions of $\mc{O}_K$. Let $\sigma \colon K \to K$ be the lift of the $p$-power Frobenius automorphism of $\kappa$. 
\item
$\ms{X}$ is a smooth separated $\mc{O}_K$-scheme with special fiber $\ms{X}_\kappa$, generic fiber $X_K \coloneqq \ms{X}_K$, and $p$-adic completion $\wh{\ms{X}}$, a smooth formal $\mc{O}_K$-scheme.
\item
$\wh{\ms{X}}_\eta$ is the rigid generic fiber of $\wh{\ms{X}}$, and $X_K^\an$ is the analytification of $X_K$.
These are smooth rigid-analytic $K$-varieties, and {$X_K^\an$ is a strict neighborhood of $\wh{\ms{X}}_\eta$. For example, when $\ms{X} = \mathbb{A}^1_{\mc{O}_K}$, $\wh{\ms{X}}_\eta$ is the closed unit disk in $X_K^\an = \mathbb{A}^{1,\an}_K$. }
\end{itemize}
We let $\Isoc(\ms{X}_\kappa)$ denote the category of convergent isocrystals on $\ms{X}_\kappa$ 
and $\Isocdag(\ms{X}_\kappa)$ the category of overconvergent isocrystals (see \cite{ogus:f-isoc-2, berthelot:cohomologie-rigide, le-stum:rigid-cohomology}). Both are rigid abelian tensor categories depending only on $\ms{X}_\kappa$. These categories can be defined for arbitrary smooth varieties over $\kappa$ but admit a simple description for varieties with a global lift to characteristic $0$: notation as above, $\Isoc(\ms{X}_\kappa)$ is realized as the full subcategory of $\MIC(\wh{\ms{X}}_\eta)$ consisting of convergent connections  \cite[Theorem 2.15]{ogus:f-isoc-2}, while an object of $\Isocdag(\ms{X}_\kappa)$ is a convergent connection on $Y$, where $Y$ is any strict neighborhood of $\wh{\ms{X}}_\eta$ in $X_K^\an$, and a morphism in $\Isocdag(\ms{X}_\kappa)$ is a morphism of connections which is defined on some common strict neighborhood of $\wh{\ms{X}}_\eta$ in $X_K^\an$. %\footnote{Note that an overconvergent isocrystal is \textit{not} merely a convergent isocrystal whose connection extends to a strict neighborhood, we must include the data of an extension to some strict neighborhood as there is not necessarily a \emph{unique} extension; a morphism of overconvergent isocrystals is a morphism of connections on some common strict neighborhood. } 
The functor $\Isocdag(\ms{X}_\kappa)\to\Isoc(\ms{X}_\kappa)$ is typically \textit{not} fully faithful, so one should be careful to distinguish between morphisms of convergent and of overconvergent isocrystals.

The category $\Isocdagq(\ms{X}_\kappa)$ has an endomorphism $F^*$ induced by the absolute Frobenius on $\ms{X}_\kappa$, and an (over)convergent \textit{$F^a$-structure} ($a \geq 1$) on $\mc{E} \in\Isocdagq(\ms{X}_\kappa)$ is an isomorphism
$\Phi\colon F^{*a}\mc{E} \xrightarrow{\sim} \mc{E}$ in $\Isocdagq(\ms{X}_\kappa)$. The category $\FaIsocdagq(\ms{X}_\kappa)$ consists of pairs $(\mc{E},\Phi)$ of an (over)convergent isocrystal and an (over)convergent $F^a$-structure, a morphism being a map of (over)convergent isocrystals which respects $F^a$-structures. Finally, $\IsocdagqF(\ms{X}_\kappa)$ is the full subcategory of $\Isocdagq(\ms{X}_\kappa)$ consisting of objects for which every irreducible subquotient admits an (over)convergent $F^a$-structure for \textit{some} $a$ \cite[Definition 5.2]{daddezio-esnault:universal-ext}. For any (over)convergent $F^a$-isocrystal on $\ms{X}_k$, the underlying (over)convergent isocrystal lies in $\IsocdagqF(\ms{X})$: the $F^a$-structure permutes the Jordan--H\"older constituents, and hence each irreducible subquotient admits an $F^{a'}$-structure for some $a' \geq a$.

\begin{eg}
\label{eg:isocrystals-over-point}
The categories $\Isocdagq(\Spec(\kappa))$ (abbreviated $\Isocdagq(\kappa)$) are both equivalent to the category of finite-dimensional $K$-vector spaces. Likewise, $\FIsocdagq(\kappa)$ are both equivalent to the categories of pairs $(V, \Phi)$ consisting of a finite-dimensional $K$-vector space $V$ and a $\sigma$-semilinear automorphism $\Phi \colon V \xrightarrow{\sim} V$. If $\kappa$ is finite and $|\kappa| = q = p^r$, then $\Phi^r$ is a $K$-linear endomorphism of $V$, which we call the ($q$-)\emph{linearized} Frobenius. We have the following compatibility with base change: if $\kappa' \supseteq \kappa$ is a finite extension, $|\kappa'| = q^s$, and $V\in\FIsoc(\kappa)$, the $q^s$-linearized Frobenius of $V \otimes_K W(\kappa')[p^{-1}] \in \FIsoc(\kappa')$ is the $W(\kappa')[p^{-1}]$-linear extension of the $s$-th power of the $q$-linearized Frobenius of $V$.  
\end{eg}

%We will need the following fact. {\color{red} PROBABLY SHOULD CUT THIS. We make the same argument later several times so not necessary to separate this out as it's own lemma. }
%
%\begin{lem}
%\label{lem:summand-overconvergent}
%Let $\mc{E} \in\IsocdagF(\ms{X}_\kappa)$.
%Then any direct summand of $\mc{E}$ in $\MIC(\wh{\ms{X}}_\eta)$ again lives in $\IsocdagF(\ms{X}_\kappa)$.
%\end{lem}
%
%\begin{proof}
%By the definition of convergence of a connection, any direct summand {\color{teal}(or even subobject)} $\mc E'$ of $\mc E$ lives in $\Isoc(\ms{X}_\kappa)$.
%According to \cite[Corollary 5.7]{daddezio-esnault:universal-ext}, the functor $\IsocdagF(\ms{X}_\kappa)\to\Isoc(\ms{X}_\kappa)$ is fully faithful, so the composition $\mc E\twoheadrightarrow\mc E'\hookrightarrow\mc E$ represents a morphism in $\IsocdagF(\ms{X}_\kappa)$.
%Hence its image $\mc E'$ represents an object of $\IsocdagF(\ms{X}_\kappa)$.
%\end{proof}
%

Finally, if $\bar{\ms{X}}_\kappa$ is a good compactification of $\ms{X}_\kappa$ (so that $\bar{\ms{X}}_\kappa$ is smooth and proper over $\kappa$, and $\ms{D}_\kappa = \bar{\ms{X}}_\kappa \setminus {\ms{X}}_\kappa$ is a strict normal crossings divisor), we write $\FIsoc(\bar{\ms{X}}_\kappa, \ms{D}_\kappa)$ for the category of convergent log-$F$-isocrystals for the log-structure determined by $\ms{D}_\kappa$ (see e.g.\ \cite[Definition 7.1]{kedlaya:notes-on-isocrystals}).
The restriction functor $\FIsoc(\bar{\ms{X}}_\kappa, \ms{D}_\kappa) \to \FIsocdag(\ms{X}_\kappa)$ is fully faithful (\cite[Theorem 6.4.5]{kedlaya:semistable-reduction}), and in particular any overconvergent $F$-isocrystal admits at most one extension on $\ms{X}_\kappa$ to a convergent log-$F$-isocrystal.
We say an overconvergent $F$-isocrystal is \emph{docile} (with respect to the compactification) if it extends to a convergent log-$F$-isocrystal that has unipotent monodromy in the sense of \cite[Definition 4.4.2]{kedlaya:semistable-reduction}; note that a docile overconvergent $F$-isocrystal is necessarily tame  \cite[Definition 2.3.2]{kedlaya:etale-crystalline-1}. In the cases we will consider, $(\bar{\ms{X}}_\kappa, \ms{D}_\kappa)$ is the special fiber of a good compactification $(\bar{\ms{X}}, \ms{D})$ of $\ms{X}$ over $\mc{O}_K$ with generic fiber $(\bar{X}_K,D_K)$. In this case, we can represent objects of $\FIsoc(\bar{\ms{X}}_\kappa, \ms{D}_\kappa)$ by objects in $\MIC(\bar{X}_K^\an, D_K^\an) \cong \MIC(\bar{X}_K, D_K)$ (the latter equivalence by rigid GAGA since $\bar{X}_K$ is proper).

In summary, we have the following commuting diagram of forgetful/restriction functors:
\[
\begin{tikzcd}
\FIsoc(\bar{\ms{X}}_\kappa, \ms{D}_\kappa) \arrow[r] & \FIsocdag(\ms{X}_\kappa)\arrow[r]\arrow[d]&
\IsocdagF(\ms{X}_\kappa)\arrow[r]\arrow[d]&
\Isocdag(\ms{X}_\kappa)\arrow[d]\\
 & \FIsoc(\ms{X}_\kappa)\arrow[r]&
\IsocF(\ms{X}_\kappa)\arrow[r]&
\Isoc(\ms{X}_\kappa).
\end{tikzcd}
\]

\subsection{Logarithmic Fontaine--Laffaille modules and Higgs--de Rham flows}\label{ss.fontaine-laffaille}
For all definitions regarding formal schemes, we refer the reader to \cite[Chapter 1]{fujiwara-kato:foundations-rigid-geometry}. Let $\kappa$ be a perfect field of characteristic\footnote{We take $p>2$ to apply results of \cite{esnault-groechenig:cristallinity} on Frobenius pullback.} $p>2$ and $W \coloneqq W(\kappa)$ the ring of $p$-typical Witt vectors with fraction field $K \coloneqq W[p^{-1}]$. By a $p$-adic formal scheme $\wh{Y}$ over $W$, we mean a formal scheme $\wh{Y}$ such that $p$ generates an ideal of definition of $\wh{Y}$, together with an adic morphism $\wh{Y} \to \Spf\, W$. We write $Y_n = \wh{Y} \times_{\Spf\, W} \Spec\, W/(p^{n+1})$, which is the scheme $(|\wh{Y}|, \mc{O}_{\wh{Y}}/(p^{n+1}))$ over $W/(p^{n+1})$.

 Let $\wh{X}$ be a smooth formal scheme over $\Spf\, W$ and $\wh{D} \subseteq \wh{X}$ a relative SNC divisor on $\wh{X}$.  This means that $\wh{D}$ is a relative SNC divisor if it is an effective Cartier divisor\footnote{An effective Cartier divisor $\wh{D}$ on $\wh{X}$ is defined as a locally principal closed ideal sheaf $\mc{I}_{\wh{D}} \subseteq \mc{O}_{\wh{X}}$; see \cite[Definition 21.1.6]{ega:IV}.} each point $x \in \wh{X}$ has an \'etale neighborhood $\wh{U}$ that is \'etale over $\Spf\, W\langle T_1, \ldots, T_d\rangle$ and $\wh{D}$ is generated by the pullback of the ideal sheaf generated by $T = T_1\cdots T_s$ for some $s$.  We let $\Omega^1_{\wh{X}}(\log\, \wh{D})$ be the sheaf of $W$-linear K\"ahler differentials that are logarithmic with respect to $\wh{D}$; this can equivalently be constructed either as the inverse limit of $\Omega^1_{X_n}(\log\, D_n)$ or via the functor of points of representing of continuous derivations from $\mc{I}_{\wh{D}}$ (following \cite[Ch.\ 1 \S 5.1 (c)]{fujiwara-kato:foundations-rigid-geometry}). The category $\MIC(\wh{X}, \wh{D})$ of coherent sheaves with $W$-linear logarithmic connections is defined verbatim as for schemes, and we write $\MIC_{\mr{lf}}(\wh{X}, \wh{D})$ for the full subcategory where the underlying coherent sheaf is locally free. 
 
We introduce two key constructions from \cite{esnault-groechenig:cristallinity} (which are implicitly defined in Faltings \cite[page 33, d)]{faltings:crys-cohom-p-adic-galois-reps},  in the non-logarithmic case). First is the Artin--Rees construction \cite[Definition 3.27]{esnault-groechenig:cristallinity}, which is a functor
 \begin{equation}\label{eqn:artin-reese}
 	\AR \colon \FilMIC(\wh{X}, \wh{D}) \to \pMIC(\wh{X}, \wh{D}).
\end{equation}
Here $\pMIC(\wh{X}, \wh{D})$ is the category of integrable $p$-connections \cite[Definition 3.8]{esnault-groechenig:cristallinity}, and $\FilMIC(\wh{X}, \wh{D}) $ consists of triples $(M, \nabla, \Fil)$ where $(M, \nabla) \in \MIC(\wh{X}, \wh{D})$ and $\Fil$ is a Griffiths-transverse filtration on $M$.
This means that $\Fil$ is a decreasing filtration that is separated and exhaustive ($\Fil^iM = 0$ if $i \gg 0$ and $\Fil^iM = M$ if $i \ll 0$), each $\Fil^i M$ is Zariski-locally on $\wh{X}$ a direct summand of $M$, and $\Fil$ satisfies Griffiths transversality: $\nabla(\Fil^i M) \subseteq \Fil^{i-1}M \otimes_{\mc{O}_{\wh{X}}} \Omega^1_{\wh{X}}(\log\, \wh{D})$. Then 	
	\[ \AR(M, \nabla, \Fil) \coloneqq \bigoplus_{i \in \ZZ} \Fil^i M/{\sim} \quad  \text{where} \quad p[v]_i \sim [v]_{i-1} \text{ for } v \in \Fil^i M,  \]
and we have written $[v]_i$ to denote that we consider $v$ placed in the $i$th graded piece. This definition is equivalent to that given in \cite[Definition 3.27]{esnault-groechenig:cristallinity} when $M$ is $p$-torsion-free, in which case $[v]_i \mapsto v \otimes p^{-i}$ identifies $\AR(M, \nabla, \Fil)$ with $\sum_{i \in \ZZ} \Fil^i M \otimes p^{-i} \subseteq M \otimes_{\ZZ_p} \QQ_p$. Using this latter description we observe that $M[p^{-1}] \cong \AR(M, \nabla, \Fil)[p^{-1}]$ when $M$ is $p$-torsion-free. The connection $\nabla^\AR$ on $\AR(M, \nabla, \Fil)$ is defined by $\nabla^\AR([v]_i) \coloneqq [\nabla(v)]_{i-1}$, which is a $p$-connection: given $f$ a local section of $\mc{O}_{\wh{X}}$ and $v \in \Fil^iM$, we have ${\nabla}^\AR(f[v]_i) = [v \otimes df + f \nabla(v)]_{i-1} = p[v]_i \otimes df + f{\nabla}^\AR[v]_i$. 

The second construction is the flow functor \cite[Definition 3.28]{esnault-groechenig:cristallinity} 
\begin{equation}\label{eqn:flow-functor}
    \Phi \colon \FilMIC_{[0, p-1]}(\wh{X}, \wh{D}) \to \MIC(\wh{X}, \wh{D})
\end{equation}
which is characterized by  a canonical isomorphism \cite[Equation (5), Def. 3.28]{esnault-groechenig:cristallinity}  $\Phi \cong \wh{F}^\ast \circ \AR$ whenever $\wh{F} \colon \wh{X} \to \wh{X}$ is a logarithmic Frobenius lift, i.e.\ a morphism of formal schemes that preserves $\wh{D}$ and whose mod-$p$ reduction is the absolute $p$-Frobenius on $X_0$. A couple of remarks are in order.
First, $\FilMIC_{[0, p-1]}(\wh{X}, \wh{D})$ consists of the full subcategory of $\FilMIC(\wh{X}, \wh{D})$ where the filtration has graded pieces concentrated in degrees $[0, p-1]$ (i.e. $\Fil^0M = M$ and $\Fil^{p}M = 0$). It is necessary to restrict to such filtrations so that the gluing isomorphism \cite[Equation (5), Def. 3.28]{esnault-groechenig:cristallinity} does not have any $p$ in the denominator (recall that the $G_j - F_j$ in the notation of \emph{loc.\ cit.\ }is divisible by $p$ since $\wh{F}, \wh{G}$ lift the same map mod $p$). Second, to define the connection $\nabla^\Phi$ on $\Phi(M, \nabla, \Fil)$, it suffices to do so when $\wh{X} = \Spf R$ is affine, $\wh{D} = V(T)$ for $T \in R$, and $\phi \colon R \to R$ is a logarithmic Frobenius lift. Identifying $M$ and $\AR(M) \coloneqq \AR(M, \nabla, \Fil)$ with their $R$-modules of global sections, $\nabla^\Phi$ corresponds to the map $(1 \otimes \frac{\phi_\ast}{p}) \circ \nabla^\AR \colon \AR(M) \otimes_{R} \to \AR(M) \otimes_{\phi} \Omega^1_R(\log\, T)$. Explicitly, given $m \in \Fil^iM$ such that $\nabla(m) = \sum_j m_j \otimes \omega_j$ for $m_j \in \Fil^{i-1}M, \omega_j \in \Omega^1_R(\log\, T)$, we have 
	\[ \nabla^\Phi([m]_i \otimes 1) \coloneqq \sum_j [m_j]_{i-1} \otimes \frac{\phi_\ast \omega_j}{p}, \]
where we write $\phi_\ast(f\cdot dg) = \phi(f)\cdot d\phi(g)$ for $f, g \in R$.
The division by $p$ here is valid as $\Omega^1(\log\, T)$ is $p$-torsion free and $\phi_\ast(\Omega^1(\log\, T)) \subseteq p\Omega^1(\log\, T)$ since $\phi$ is a Frobenius lift. 

We can finally introduce the category of logarithmic Fontaine--Laffaille modules $\MF_{[0,p-1]}^\nabla(\wh{X}, \wh{D})$. The objects consist of quadruples $(M, \nabla, \Fil, \varphi_M)$ where 
	\[ (M, \nabla, \Fil) \in \FilMIC(\wh{X}, \wh{D}) \text{\quad and \quad}\varphi_M \colon \Phi(M, \nabla, \Fil) \xrightarrow{\sim} M, \]
the latter being an isomorphism in $\MIC(\wh{X}, \wh{D})$. Morphisms are defined as those in $\MIC(\wh{X}, \wh{D})$ that preserve the filtration and the Frobenius structure. We also introduce the full subcategory $\MF_{[0,p-1], \mr{lf}}(\wh{X}, \wh{D})$ where the underlying coherent sheaf is locally free. 

\begin{rmk}
The category $\MF_{[0,p-1]}^\nabla(\wh{X}, \wh{D})$ agrees with Faltings's original category \cite[page 33, d)]{faltings:crys-cohom-p-adic-galois-reps} when $\wh{D}$ is empty. It also agrees with the category of logarithmic Fontaine--Faltings modules from \cite[\S 2]{liu-yang-zuo:log-crystalline}; one sees this by comparing the gluing isomorphisms of \cite[Definition 3.28 (b)]{esnault-groechenig:cristallinity} and \cite[Eqn.\ (2.2)]{liu-yang-zuo:log-crystalline}. {It seems there is a small typo in \cite[Eqn.\ (2.2)]{liu-yang-zuo:log-crystalline}, where the first ``$\min$'' should be a ``$\max$'', as one should recover the formula used in the proof of \cite[Theorem 2.3]{faltings:crys-cohom-p-adic-galois-reps}.}
\end{rmk}

\begin{rmk}\label{rmk:MF-to-FIsoc}
If $M \in \MF^{\nabla}_{[0,p-1], \mr{lf}}(\wh{X}, \wh{D})$, then $\AR(M)[p^{-1}] = M[p^{-1}]$ and thus $M[p^{-1}]$ is a (filtered) $F$-isocrystal. Thus inverting $p$ gives a functor $\MF^{\nabla}_{[0,p-1]}(\wh{X}, \wh{D}) \to \FIsoc(X_0, D_0)$. 
\end{rmk}
 
% The Frobenius pullback functor \cite[Proposition 3.3]{esnault-groechenig:cristallinity} 
% \begin{equation}\label{eqn:frob-pullback}
% 	F^\ast \colon \MIC(\wh{X}, \wh{D}) \to \MIC(\wh{X}, \wh{D})
% \end{equation}
% is characterized by a canonical isomorphism $F^\ast \cong \wh{F}^\ast$ whenever $\wh{F} \colon \wh{X} \to \wh{X}$ is a morphism of formal schemes that preserves $\wh{D}$ and whose mod $p$ reduction is the absolute $p$-Frobenius on $X_0$.

\begin{rmk}\label{rmk:HdR-flow-explicit}
A category of $f$-periodic logarithmic Higgs--de Rham flows is not explicitly defined in \cite{esnault-groechenig:cristallinity}, but can easily be defined using $\Phi$. An $f$-periodic logarithmic Higgs--de Rham flow on $(\wh{X}, \wh{D})$ consists of the data $M = (M_{(0)}, \nabla_0, \Fil_0, \Fil_1, \ldots, \Fil_{f-1}, \phi)$ where 
\begin{itemize}
    \item $(M_{(0)}, \nabla_0)\in\MIC_{\lf}(\wh{X}, \wh{D})$;
    \item $\Fil_i$ for $i = 0,\ldots, f -1$ is a filtration on $(M_{(i)}, \nabla_i) \coloneqq \Phi(M_{(i-1)}, \nabla_{i-1}, \Fil_{i-1})$  which is Griffiths-transverse; and 
    \item $\varphi_M \colon \Phi(M_{(f-1)}, \nabla_{f-1}, \Fil_{f-1}) \xrightarrow{\sim} (M_{(0)}, \nabla_0)$ is an isomorphism in $\MIC(\wh{X}, \wh{D})$. 
\end{itemize} 
A morphism $M \to N$ is a morphism $M_{(0)} \to N_{(0)}$ of modules with integrable connection that preserves all filtrations (after applying appropriate iterates of $\Phi$). We write $\HdR_{[0,p-1]}^f(\wh{X}, \wh{D})$ for the category of $f$-periodic Higgs--de Rham flows such that each filtration is concentrated in degrees $[0,p-1]$. 
With these definitions, $\HdR_{[0,p-1]}^1(\wh{X}, \wh{D})$ is \emph{equal} to $\MF_{[0,p-1], \mr{lf}}^{\nabla}(\wh{X}, \wh{D})$. When $\mathbb{Z}_{p^f} := W(\mathbb{F}_{p^f}) \subseteq W$, the category $\HdR^f_{[0,p-1]}(\wh{X}, \wh{D})$ is equivalent to  $\MF^{\nabla}_{[0,p-1], \mr{lf}}(\wh{X}, \wh{D}) \otimes_{\ZZ_p} \mathbb{Z}_{p^f}$ the latter category being the scalar extension of the $\ZZ_p$-linear category to a $\mathbb{Z}_{p^f}$-linear category (objects are Fontaine--Laffaille modules with endomorphism structure by $\mathbb{Z}_{p^f}$, and morphisms respect the $\ZZ_{p^f}$-structure). The functor sends $M = (M_{(0)}, \nabla_0, \Fil_0, \ldots, \Fil_{f-1})$ to $N\coloneqq M_{(0)} \oplus \cdots \oplus M_{(f-1)}$ with the direct sum of filtrations/connections, $\varphi_N$ the direct sum of the identifications $\Phi(M_{(i)}, \nabla, \Fil_i) = M_{(i+1)}$ using $\varphi_M$ when $i = f-1$, and the $\mathbb{Z}_{p^f}$-endomorphism structure determined by the action of a primitive $(p^f-1)$-th root of unity $\zeta \in \ZZ_{p^f}$ which acts on $M_{(i)}$ by $\zeta^{p^i}$.
\end{rmk}

\section{Companions}\label{s.companions}
Let $\kappa$ be a finite field of characteristic $p$, let $K \coloneqq W(\kappa)[p^{-1}]$, and let $X$ be a smooth, geometrically connected scheme over $\kappa$. Let $\ol{\QQ}$ be an algebraic closure of $\QQ$, and for a finite place $\lambda$ of $\ol{\QQ}$ dividing a rational prime $\ell$, we write $\ol{\QQ}_\lambda$ for the algebraic closure of $\QQ_\ell$ in the completion of $\ol{\QQ}$ at $\lambda$. Let $\pi$ be a finite place of $\ol{\QQ}$ dividing $p$. % Fix an embedding $\iota_0 \colon K \to \ol{\QQ}_\pi$ for this section. 

Let $\FIsocdag(X)_{\ol{\QQ}_\pi} \coloneqq \FIsocdag(X) \otimes_{\QQ_p} {\ol{\QQ}_\pi}$ be the $\ol{\QQ}_\pi$-linearization:
    \[ \FIsocdag(X)_{\ol{\QQ}_\pi} = 2\text{-}\mr{colim}_{E \subseteq \ol{\QQ}_\pi} \FIsocdag(X)_E \]
where $E$ ranges over finite extensions of $\QQ_p$ in $\ol{\QQ}_\pi$, and the linearization is as in \cite[\S 3]{deligne-milne}. 

\begin{eg}\label{eg:Qpbar-isocrystals-over-points}
If $X = \Spec(\kappa)$, then objects in $\FIsoc(\kappa)_{\ol{\QQ}_\pi}$ are pairs $(W, \Phi)$ consisting of a finite $K \otimes_{\QQ_p} \ol{\QQ}_\pi$-module $W$ with an $\sigma \otimes 1$-semilinear isomorphism $\Phi \colon W \xrightarrow{\sim} W$. For $\iota \colon K \hookrightarrow \ol{\QQ}_\pi$, let $e_\iota \in K \otimes_{\QQ_p} \ol{\QQ}_\pi$ be the element corresponding to projection onto the $\iota$-component of  $\prod_{\iota \colon K \hookrightarrow \ol{\QQ}_\pi} \ol{\QQ}_\pi \cong K \otimes_{\QQ_p} \ol{\QQ}_\pi$. If we fix a single embedding $\tau \colon K \hookrightarrow \ol{\QQ}_\pi$, the submodule $e_\tau W \subseteq W$ is a $\ol{\QQ}_\pi$-subspace, which is stable under $\Phi^r$, where $r = [\kappa : \mbb{F}_p]$. The assignment $(W, \Phi) \mapsto (e_\tau W, \Phi^r)$ gives an equivalence of $\FIsoc(\kappa)_{\ol{\QQ}_\pi}$ with the category of finite dimensional $\ol{\QQ}_\pi$-vector spaces with $\ol{\QQ}_p$-\emph{linear} automorphisms. (A quasi-inverse functor sends $V$ to $\mr{Ind}_{\langle \sigma^r\rangle}^{\langle \sigma \rangle} V$, where $\sigma^r$ acts by the distinguished automorphism.) The $\ol{\QQ}_\pi$-linear automorphism is exactly the $|\kappa|$-linearized Frobenius of \cref{eg:isocrystals-over-point}. 
\end{eg}

The category $\FIsocdag(X)_{\ol{\QQ}_\pi}$ is a neutral Tannakian category over $\ol{\QQ}_\pi$.\footnote{
The same holds for $\FIsoc(X)_{\ol\QQ_\pi}$, but in this section we will only focus on the overconvergent $F$-isocrystals, as these are objects that have companions.
}
Indeed, fix a `geometric point' $\bar{x} = (x, \tau)$, i.e.\ $x \in X(\kappa')$ for $\kappa'$ a finite extension of $\kappa$, and an embedding $\tau \colon W(\kappa')[p^{-1}] \to \ol{\QQ}_\pi$. The composition (with the last map as in \cref{eg:Qpbar-isocrystals-over-points})
    \[ \omega_{\bar{x}} \colon \FIsocdagq(X) \xrightarrow{x^\ast} \FIsoc(\kappa') \to \FIsoc(\kappa')_{\ol{\QQ}_\pi} \to \Vect_{\ol{\QQ}_\pi} \]
is a fiber functor. We let $\pi_1^{\FIsocdag}(X, \bar{x}) = \ul{\Aut}^\otimes(\omega_{\bar{x}})$ be the associated Tannakian group.
This is an affine algebraic group over $\ol{\QQ}_\pi$, and $\omega_{\bar{x}}$ induces an isomorphism $\FIsocdag(X)_{\ol{\QQ}_\pi} \xrightarrow{\sim} \Rep(\pi_1^{\FIsocdag}(X,\ol{x}))$ \cite[Theorem 2.11]{deligne-milne}.

\begin{rmk}
We sometimes write $\pi_1^{\FIsocdag}(X)$ without a base point. In such situations, all claims are independent of base points. 
\end{rmk}

\begin{defn}
\label{defn:G-F-isocrystal}
Let $G$ be a connected linear algebraic group over $\QQ_p$. An \emph{overconvergent $G$-$F$-isocrystal} on $X$ is an exact $\QQ_p$-linear tensor functor $\mc{E} \colon \Rep(G) \to \FIsocdagq(X)$. We write $\mc{E}_{\ol{\QQ}_\pi} = \mc{E} \otimes_{\QQ_p} \ol{\QQ}_\pi \colon \Rep(G) \to \FIsocdagq(X) \otimes_{\QQ_p} \ol{\QQ}_\pi$ for the $\ol{\QQ}_\pi$-linear extension of $\mc{E}$ (an overconvergent $G$-$F$-isocrystal with coefficents in $\ol{\QQ}_\pi$). We denote the Tannakian dual of $\mc{E}_{\ol{\QQ}_\pi}$ by 
    \[ \rho_{\mc{E}} \colon \pi_1^{\FIsocdag}(X) \to G_{\ol{\QQ}_\pi}, \]
which is well defined up to $G(\ol{\QQ}_\pi)$-conjugacy. 
\end{defn}

We next define the Frobenius conjugacy class of $\mc{E}$ at a point $x \in X(\kappa')$ for $\kappa'$ a finite extension of $\kappa$.
Choose $\tau \colon W(\kappa') \to \ol{\QQ}_\pi$ and let $\bar{x} = (x, \tau)$. Then $\omega_{\bar{x}} \circ \mc{E}$ is a fiber functor $\Rep(G) \to \Vect_{\ol{\QQ}_\pi}$, and the $|\kappa'|$-linearized Frobenius automorphism of \cref{eg:Qpbar-isocrystals-over-points} gives a distinguished element in $\ul{\Aut}^\otimes(\omega_{\bar{x}} \circ \mc{E})$.
By \cite[Theorem 3.2]{deligne-milne}, there is an isomorphism $\ul{\Aut}^\otimes(\omega_x \circ \mc{E}) \cong G_{\ol{\QQ}_\pi}$ (unique up to $G(\ol{\QQ}_\pi)$-conjugation), we let $\rho_{\mc{E}}(\Frob_x) \subseteq G(\ol{\QQ}_\pi)$ be the $G(\ol{\QQ}_\pi)$-conjugacy class given by the image of the $|\kappa'|$-linearized Frobenius in $\ul{\Aut}^\otimes(\omega_{\bar{x}} \circ \mc{E})$.
(This is independent of the choice of $\tau$.)
The Frobenius conjugacy class gives a well defined semisimple conjugacy class in $[G\git G](\ol{\QQ}_\pi)$.

\begin{defn}\label{defn:weak-companions}
Let $G$ be a connected linear algebraic group over $\ol{\QQ}$. Let $\mc{E}$ be an overconvergent $G$-$F$-isocrystal on $X$ with coefficients in $\ol{\QQ}_\pi$. Let $\lambda$ be a finite place of $\ol{\QQ}$ lying over a rational prime $\ell \neq p$, and suppose $\rho_{\lambda} \colon \pi_1(X) \to G(\ol{\QQ}_\lambda)$ is a $G(\ol{\QQ}_\lambda)$-local system.
We say that $\mc{E}$ and $\rho_{\lambda}$ are \emph{weak companions} if for all closed points $x \in X$, the semisimple conjugacy classes of $\rho_{\lambda}(\Frob_x)$ and $\rho_{\mc{E}}(\Frob_x)$ (in $[G\git G](\ol{\QQ}_\lambda)$ and $[G\git G](\ol{\QQ}_\pi)$, respectively) are in fact defined over $\ol{\QQ}$ and are equal. 
\end{defn}

\begin{rmk}
The condition that $\rho_{\lambda}$ and $\mc{E}$ are weak companions could equivalently be stated as saying that for each $x\in X(\kappa')$ with $\kappa'$ a finite extension of $\kappa$ and each $\xi\in\Rep(G)$, the characteristic polynomials of $(\xi\circ\rho_\lambda)(\Frob_x)$ and of the linearized Frobenius action on the fiber $x^*\mc{E}(\xi)$ both have coefficients in $\ol{\QQ}$ and are equal.
\end{rmk}

\begin{rmk}\label{rmk:uniqueness-weak-companions}
Weak $\lambda$-adic companions of an overconvergent $G$-$F$-isocrystal on $X$ need not be unique (particularly when the monodromy group is not connected). A better notion of companion, which ensures uniqueness, is given in \cref{rmk:strong-companions}. However, for our main application (\cref{thm:companions}), the two notions agree: if $G$ is a connected semisimple group over $\ol{\QQ}$, and $\mc{E}$ is a $G$-$F$-overconvergent isocrystal with $\ol{\QQ}_\pi$-coefficients such that the monodromy representation $\rho_{\pi} \colon \pi_1^{\FIsocdag}(X) \to G_{\ol{\QQ}_\pi}$ is surjective (i.e.\ faithfully flat), then $\lambda$-adic companions are unique. Indeed, the strong $\lambda$-adic companion $\rho_{\pi \leadsto \lambda}$ of $\rho_{\pi}$ from \cref{rmk:strong-companions} has Zariski-dense image, and since $G$ is connected, \cite[Proposition 6.4]{bhkt:fnfieldpotaut} implies that any other weak $\lambda$-adic companion of $\rho_\pi$ is $G(\ol{\QQ}_\lambda)$-conjugate to $\rho_{\pi \leadsto \lambda}$. 
\end{rmk}

The following theorem is a combination of work of Drinfeld, Kedlaya, Abe, and Abe--Esnault. 

\begin{thm}
\label{thm:companions}
Let $G$ be an algebraic group over $\ol{\QQ}$ with semisimple neutral component and $\mc{E}$ an overconvergent $G$-$F$-isocrystal on $X$ with coefficients in $\ol{\QQ}_\pi$ such that the monodromy representation $\rho_{\pi} \colon \pi_1^{\FIsocdag}(X) \to G_{\ol{\QQ}_\lambda}$ of $\mc{E}$ is surjective.\footnote{equivalently, that the functor $\mc{E}$ is fully faithful and its essential image is closed under taking subobjects \cite[Proposition 2.21]{deligne-milne}}
Then $\mc{E}$ admits a weak $\lambda$-companion $\rho_{\pi \leadsto \lambda} \colon \pi_1(X)\to G(\ol{\QQ}_\lambda)$ for any finite place $\lambda$ of $\ol{\QQ}$ not above $p$.
Moreover:
\begin{enumerate}
\item
The image of $\rho_{\pi \leadsto \lambda}$ is Zariski-dense in $G_{\ol{\QQ}_\lambda}$.
\item
If $\mc{E}$ takes values in the subcategory of tame overconvergent $F$-isocrystals, then $\rho_{\pi \leadsto \lambda}$ factors through the tame fundamental group of $X$.
\end{enumerate}
\end{thm}

\subsection{Proof of \cref{thm:companions}}

If $\dim(X) = 1$, statement (1) follows immediately from \cite[Theorem 7.5.1]{drinfeld:pross}. We will explain why the conclusion of \cite[Theorem 7.5.1]{drinfeld:pross} remains true when $\dim(X) > 1$.
(See \cite[\S4.2]{drinfelds-lemma} for a parallel discussion.)

Recall (\cite[\S\S1.1--1.2 and 7.1--7.3]{drinfeld:pross}) that Drinfeld constructs affine pro-algebraic groups $\wh{\Pi}_{(\lambda)}$ and $\wh{\Pi}_{(\pi)}$ over $\ol{\QQ}$. By construction and \cite[Proposition 2.3.3]{drinfeld:pross}, these have following property:
If $G$ is a linear algebraic group over $\ol{\QQ}$ with semisimple neutral component, there exist ``canonical'' bijections between
\begin{itemize}
\item
$G^\circ(\ol{\QQ})$-conjugacy classes of surjective homomorphisms $r_\lambda \colon \wh\Pi_{(\lambda)} \twoheadrightarrow G$ and $G^\circ(\ol{\QQ}_\lambda)$-conjugacy classes of continuous homomorphisms $\rho_\lambda \colon \pi_1(X)\to G(\ol{\QQ}_\lambda)$ with Zariski-dense image.
\item
$G^\circ(\ol{\QQ})$-conjugacy classes of epimorphisms $r_\pi\colon \wh{\Pi}_{(\pi)} \twoheadrightarrow G$ and $G^\circ(\ol{\QQ}_{\ol{\pi}})$-conjugacy classes of surjective homomorphisms $\rho_\pi \colon \pi_1^{\FIsocdag}(X)\twoheadrightarrow G_{\ol{\QQ}_\pi}$.
\end{itemize}
Moreover, there exists a diagram (\cite[\S\S 1.3 and 7.3--7.4]{drinfeld:pross})
\begin{equation}
\label{eq:drinfeld-diagram}
\begin{tikzcd}
[row sep=0em]
&{[\wh{\Pi}_{(\lambda)}]}(\ol{\QQ})\arrow[rd]\\
\Pi_{\Frob}\arrow[ru]\arrow[rd]&&\pi_1(X)\\
&{[\wh{\Pi}_{(\pi)}]}(\ol{\QQ})\arrow[ru]
\end{tikzcd}
\end{equation}
in which $\Pi_{\Frob}\subseteq\pi_1(X)$ is the union of all Frobenius conjugacy classes $\Frob_x$ as $x$ ranges over all points of $X$ valued in finite fields, $\Pi_{\Frob} \to \pi_1(X)$ is the inclusion map, and $[H]$ denotes the GIT quotient $H\git H^\circ$ by the conjugation action.
This diagram has the following property:
For any $\gamma \in \Pi_{\Frob}$,
\begin{itemize}
\item
If $r_\lambda \colon \wh{\Pi}_{(\lambda)} \twoheadrightarrow G$ and $\rho_\lambda \colon \pi_1(X)\to G(\ol{\QQ}_\lambda)$ correspond under the bijection above, then $[\rho_\lambda(\gamma)]$ lives in $[G](\ol{\QQ})$ and equals $[r_\lambda(\gamma)]$.\footnote{
This property visibly does not depend on which $G^\circ$-conjugates of $r_\lambda$ and $\rho_\lambda$ we pick; likewise in the second bullet.
}
\item
If $r_\pi \colon \wh{\Pi}_{(\pi)} \twoheadrightarrow G$ and $\rho_\pi \colon \pi_1^{\FIsocdag}(X) \twoheadrightarrow G_{\ol{\QQ}_\pi}$ correspond under the bijection above, then $[\rho_\pi(\gamma)]$ lives in $[G](\ol{\QQ})$ and equals $[r_\pi(\gamma)]$.
\end{itemize}

Now we state \cite[Theorem 7.5.1]{drinfeld:pross} for $\dim(X)$ arbitrary.

\begin{thm}
\label{thm:drinfeld-751}
There exists an isomorphism $\wh{\Pi}_{(\lambda)}\xrightarrow{\sim}\wh{\Pi}_{(\pi)}$ of affine group schemes over $\ol{\QQ}$, unique up to conjugation by $\wh{\Pi}_{(\pi)}^\circ(\ol{\QQ})$, such that the induced map $[\wh{\Pi}_{(\lambda)}] \xrightarrow{\sim} [\wh{\Pi}_{(\pi)}]$ makes diagram \eqref{eq:drinfeld-diagram} commute.
\end{thm}

\begin{rmk}\label{rmk:strong-companions}
The above setup allows us to make precise the notion of a strong companion of an overconvergent $F$-isocrystal with semisimple (not necessarily connected!)\ monodromy group:
Let $G$ be an affine algebraic group over $\ol{\QQ}$ such that $G^\circ$ is semisimple, and $\mc{E}$ an overconvergent $G$-$F$-isocrystal with $\ol{\QQ}_\pi$-coefficients such that the monodromy representation $\rho_{\pi} \colon \pi_1^{\FIsocdag}(X) \to G_{\ol{\QQ}_\pi}$ is surjective.
This is equivalent to the data of a surjection $\hat{\Pi}_{(\pi)} \twoheadrightarrow G$ of algebraic groups over $\ol{\QQ}$ (unique up to $G^\circ(\ol{\QQ})$-conjugacy). The composition 
    \[ \wh{\Pi}_{(\lambda)}\xrightarrow{\sim}\wh{\Pi}_{(\pi)} \twoheadrightarrow G \]
then corresponds to a continuous homomorphism $\rho_{\pi \leadsto \lambda} \colon \pi_1(X) \to G(\ol{\QQ}_\lambda)$, which we call the \emph{strong} $\lambda$-adic companion of $\mc{E}$ (or $\rho_\pi$). This is unique up to $G(\ol{\QQ}_\lambda)$-conjugation, and the properties above ensure that it is a weak $\lambda$-adic companion of $\mc{E}$ in the sense of \cref{defn:weak-companions}.
\end{rmk}

\begin{proof}[Proof of \cref{thm:drinfeld-751}]
The argument is the same as that of \cite[\S5]{drinfeld:pross}, writing $\pi$ in place of $\lambda'$; the following comments should alleviate any concerns that might arise in making this replacement.
\begin{itemize}
\item
By (both directions of) the crystalline-to-\'etale companions correspondence \cite{kedlaya:etale-crystalline-1}, \cite{kedlaya:etale-crystalline-2}, the discussion of \cite[\S\S5.1--5.2 up to the statement of Theorem 5.2.1]{drinfeld:pross} holds verbatim with $\pi$ in place of $\lambda'$.
\item
After the statements of \cite[Theorem 5.2.1 and Remark 5.5.4]{drinfeld:pross}, Drinfeld's arguments require that $\wh{\Pi}_{(\pi)}^\circ$ and $\wh{\Pi}_{(\lambda)}^\circ$ be products of almost simple groups.
For $\wh{\Pi}_{(\lambda)}^\circ$, this follows from \cite[Proposition 3.3.4]{drinfeld:pross} (which in fact proves that $\wh{\Pi}_{(\lambda)}^\circ$ is simply connected).
The corresponding result for $\wh{\Pi}_{(\pi)}^\circ$ follows from \cite[Theorem 4.3.7, (iii)]{drinfeld:pross} plus the isomorphism \cite[equation (5.3)]{drinfeld:pross}.
\item
The proof of \cite[Lemma 5.5.8]{drinfeld:pross} only applies to $\wh{\Pi}_{(\lambda)}$.
But this is immaterial for the argument of \cite[\S5.5.9]{drinfeld:pross} due to \cite[Lemma 5.5.6, (i)$\implies$(ii)]{drinfeld:pross}.\qedhere
\end{itemize}
\end{proof}

Finally, we explain how to deduce Theorem \ref{thm:companions}.
Let $\mc{E}$ be an overconvergent $G$-$F$-isocrystal on $X$ with surjective monodromy representation $\rho_{\pi} \coloneqq \rho_{\mc{E}} \colon \pi_1^{\FIsocdag}(X) \to G_{\ol{\QQ}_\pi}$. Let $\rho_{\pi \leadsto \lambda}$ be the strong companion of \cref{rmk:strong-companions}. By construction, property (1) of Theorem \ref{thm:companions} holds, and (2) follows from \cite[Corollary 3.15]{kedlaya:etale-crystalline-1}. 

\begin{rmk}
That \cite[Theorem 7.5.1]{drinfeld:pross} holds when $\dim(X)>1$ also answers \cite[Question 7.5.3]{drinfeld:pross}: $\wh{\Pi}_{(\pi)}^\circ$ and, therefore, $\pi_1^{\FIsocdag}(X)^\circ$ are simply connected.
\end{rmk}

%Let $\mc{T}_{p}(\ms{X}_\kappa)$ be the full subcategory of semisimple objects of $\FIsoc(\ms{X}_\kappa)$, and $\mc{T}_\pi(\ms{X}_\kappa) = \mc{T}_{p}(\ms{X}_\kappa) \otimes_{\QQ_p} \ol{\QQ}_\pi$.  By \cite[Theorem 2.11]{deligne-milne}, $\omega$ induces an equivalence of categories $\mc{T}_\pi(\ms{X}_\kappa) \xrightarrow{\sim} \Rep\, \mr{Aut}^\otimes( \omega|_{\mc{T}_\pi(\ms{X}_\kappa)})$. The group $\hat{\Pi}_\pi = \mr{Aut}^\otimes( \omega|_{\mc{T}_\pi(\ms{X}_\kappa)})$ depends on the choice of fiber functor $\omega$, which in turn depends on a choice of geometric point $\tilde{x}$ of $\widetilde{\ms{X}}_\kappa$. As explained in \cite[\S 7.3.3]{drinfeld:pross}, the group $\mr{Aut}^\otimes( \omega|_{\mc{T}_\pi(\ms{X}_\kappa)})$ gives a well defined $\hat{\Pi}_\pi$ object of $\Pross(\ol{\QQ}_\pi)$, the category of pro-semisimple (not necessarily connected) group over $\ol{\QQ}_\pi$ up to conjugation by the $\ol{\QQ}_\pi$-points of the connected component of the identity. From  \cite[Proposition 2.2.5]{drinfeld:pross} there is a group $\hat{\Pi}_{(\pi)} \in \Pross(\ol{\QQ})$ whose base change to $\ol{\QQ}_\pi$ is isomorphic to $\hat{\Pi}_\pi$ (in $\Pross(\ol{\QQ}_\pi)$). 

\section{Statement of main theorem}\label{s.main-thm-statement}
\subsection{Assumptions on the Shimura datum}
Let $(G,X)$ be a Shimura datum such that $Z_G(\QQ)$ is discrete in $Z_G(\mbb{A}_f)$.\footnote{
This implies, as noted in the introduction, that $G=G^c$ in the notation of \cite[\S5]{diao-lan-liu-zhu:log-riemann-hilbert} (see \cite[Theorem 5.26]{milne:shimuraintro}).
}
We let $\Sh = \Sh(G,X)$ be the Shimura variety associated to $(G,X)$, defined over the reflex field $E(G,X) \subseteq \CC$. By assumption on $(G,X)$ we have $\Sh(\CC) = G(\QQ)\backslash X \times G(\mbb{A}_f)$, and if $x \in X, g \in G(\mbb{A}_f)$ we write $[x,g]$ for the image of $(x,g)$ in $\Sh(\CC)$. Let $s = [x,g] \in \Sh(\CC)$. For any neat compact open subgroup $K \subseteq G(\mbb{A}_f)$, we let $\Sh_K = \Sh_K(G,X)$ be the Shimura variety of level $K$, also defined over $E(G,X)$.
We let $sK \in \Sh_K(\CC)$ be the image of $s$ under $\Sh(\CC) \to \Sh_K(\CC)$, and we let $S_{K,s,\ol\QQ}$ be the connected component of $S_{K,\ol\QQ}$ that contains $sK$.
This component defined over a number field $E_{K,s}$, namely the fixed field of the stabilizer in $\Gal(\ol{\QQ}/E(G,X))$ of the connected component of $\pi_0(\Sh_{K, \ol{\QQ}})$ containing $s$.
We write $S_{K,s}$ for the descent of $S_{K,s,\ol\QQ}$ to $E_{K,s}$; it is a geometrically connected component of $\Sh_{K,E_{K,s}}$.
If $X^+$ is the connected component of $X$ containing $x$, then $\Sh_{K,s}(\CC) \cong \Gamma_{K,s} \backslash X^+$ for $\Gamma_{K,s} = G(\QQ)_+ \cap gKg^{-1}$ (where $G(\QQ)_+$ is the set of elements of $G(\QQ)$ mapping to the identity component of $G^\ad(\RR)$ in the analytic topology).

Let $\Sh^{\mr{tor}}_K(G,X)_\CC$ be a toroidal compactification of $\Sh_K(G,X)_\CC$ \cite{ash-mumford-rapoport-tai:toroidal} (see also \cite{harris-compactification} or \cite{pink-compactification}). 
We may choose the compactification so that the following hold:

\begin{enumerate}
    \item
    It has a canonical model $\Sh^\tor_K(G,X)$ over $E(G,X)$ \cite[Proposition 2.8]{harris-compactification}. 
    \item
    The toroidal compactification $\Sh^\tor_K(G,X)$ is a \textit{good compactification} in the sense that it is smooth and proper over $E(G,X)$ (even projective), and the boundary divisor has strict normal crossings \cite[Remark 2.5.7, (a)]{harris-compactification}. 
\end{enumerate}
We let $\bar{S}_{K,s}$ be the closure of $S_{K,s}$ in $\Sh^\tor_{K, E_{K,s}}(G,X)$, which is a good compactification of $S_{K,s}$. Moreover $\bar{S}_{K,s}$ is geometrically connected because it contains the geometrically connected dense open subvariety $S_{K,s}$. 

Finally, we fix once and for all a neat compact open subgroup $K_0 = \prod_{\ell} K_{0,\ell} \subseteq G(\mbb{A}_f)$, a point $s = [x,g]$ such that $sK_0 \in \Sh_{K_0}(E)$, for a finite extension $E$ of $E_{K_0, s}$, and write $\bar{s} \in \Sh_{K_0}(\CC)$ for the geometric point lying over $sK_0$. We then let $S = S_{K_0, s}, \bar{S} = \bar{S}_{K_0,s}$ considered as (geometrically connected) varieties over $E$, and $\Gamma = gK_0g^{-1} \cap G(\QQ)_+$ so that $\Gamma = \pi_1^\top(S_\CC, \bar{s})$.\footnote{A priori, $\pi_1^\top(S_\CC, \bar{s})$ is identified with the image $\Gamma^\ad$ of $\Gamma$ under the quotient $G(\QQ) \to G^\ad(\QQ)$. But $\Gamma \cong \Gamma^\ad$: the kernel $Z_G(\QQ) \cap gK_0g^{-1}$ is a finite (hence trivial as $gK_0g^{-1}$ is neat) subgroup since $Z_G(\QQ) \subseteq Z_G(\mbb{A}_f)$ is discrete and $gK_0g^{-1}$ compact.}
%We also fix a faithful representation $\xi_0$ of $G$.
%Finally, after replacing $E$ by a finite extension, we can assume the following:

%\begin{assumption}
%Every vector bundle with integrable connection that is rigid and of rank $\leq \mr{rank}(\xi_0)$  on $S_\CC$ is defined over $E$. 
%\end{assumption}

\begin{assumption}\label{assumption:superrigid}
Every $\QQ$-simple factor of $G^\ad$ has real rank $\geq 2$.
\end{assumption}

This assumption is key to later arguments, as it implies that $\Gamma=\pi_1^\mr{top}(S(\CC), \bar{s})$ is $G$-superrigid \cite[Definition 2.2, Proposition 3.3]{patrikis-klevdal:shimura}). This means that if $\Omega$ is an algebraically closed field of characteristic 0 and $\rho_1, \rho_2 \colon \Gamma \to G^\ad(\Omega)$
are homomorphisms with Zariski-dense image, there exists $\tau \in \Aut(G_\Omega^\ad)$ such that $\rho_1 = \tau \circ \rho_2$ (and, additionally, that $\Hom(\Gamma, G/G^\mr{der}(\Omega))$ is finite).

\subsection{Canonical local systems}\label{ss.canonical-local-systems}
We give the description of the canonical $G(\QQ_\ell)$-local systems on our geometrically connected Shimura variety $S$. There is a canonical homomorphism 
    \[ \rho_{K_0,s} \colon \pi_1(S, \bar{s}) \to K_0 = \varprojlim_{K \leq K_0} K_0/K \]
where the limit is taken over normal compact open subgroups $K$ of $K_0$ and the the projection $\pi_1(S, \bar{s}) \to K_0/K$ is given by the Galois cover $\widetilde{\Sh}_{K, s}/S$ with $\widetilde{\Sh}_{K,s}$ the connected component of $\Sh_K \otimes_{E(G,X)} E$ containing $sK$ (see \cite[\S3.1]{patrikis-klevdal:shimura}). For each prime $\ell$, we have the projection
    \[ \rho_\ell \coloneqq \rho_{K_0, s, \ell} \colon \pi_1(S, \bar{s}) \to K_0 \to K_{0,\ell} \subset G(\QQ_\ell), \]
and their adjoint quotients
    \[ \rho_\ell^\ad \coloneqq \rho_{K_0, s, \ell}^\ad \colon \pi_1(S, \bar{s}) \to G^\ad(\QQ_\ell). \]

\subsection{Canonical flat $G$-bundles}\label{ss.canonical-g-bundles}
We give the description of the canonical flat $G$-bundles on $S$. Let $(\xi, V) \in \Rep(G)$. The representation 
    \[ \pi_1^\top(S_\CC, \bar{s}) = \Gamma \subseteq G(\QQ) \xrightarrow{\xi} \GL(V) \]
gives rise to a $\QQ$-local system $V(\xi)$ on $S_\CC^\an$. We have an associated analytic flat vector bundle
    \[ \mc{V}(\xi)^\an = V(\xi) \otimes_\QQ \mc{O}_{S_\CC^\an}, \qquad \nabla = 1 \otimes d \]
on $S_\CC^\an$ such that $\mc{V}(\xi)^{\an, \nabla = 0} = V(\xi) \otimes_\QQ \CC$. By \cite[Theorem 4.2 and its proof]{harris-compactification}, there is a canonical descent and extension of $(\mc{V}(\xi)^\an, {\nabla})$ to an object
    \[ (\bar{\mc{V}}(\xi), \bar\nabla) \in \MIC^\nilp(\bar{S}, S) \]
such that the analytification of $\bar{\mc{V}}(\xi)|_{S_\CC}$ is $\mc{V}(\xi)^\an$, and same for the connection (the fact that the residues are nilpotent is discussed in \cite[\S 6.1]{lan-suh:vanishing-general-PEL}). 

The construction is functorial in $\Rep(G)$, and induces an exact $\QQ$-linear tensor functor 
    \[ \bar{\mc{V}} \colon \Rep(G) \to \MIC^\nilp(\bar{S}, S). \]
If $v$ is a finite place of $E$ lying over the rational prime $p$, then there is a unique $\QQ_p$-linear extension of the above tensor functor
    \[ \bar{\mc{V}}_{E_v} \colon \Rep(G_{\QQ_p}) \to \MIC^\nilp(\bar{S}_{E_v}, S_{E_v}). \]
One way to construct $\bar{\mc V}_{E_v}$ is to use the equivalence between tensor functors and torsors (see e.g.\ \cite[Theorem 4.8]{broshi}): view $\bar{\mc V}$ as a $G$-torsor $\bar P\to\bar S$, and then view the $G$-torsor $\bar P_{E_v}\to\bar S_{E_v}$ as a tensor functor.
Note that $\bar{P}$ is the canonical model of the standard principal bundle in the terminology of \cite[Theorem 4.3]{milne:canonical-models}.

\subsection{$p$-adic compatibilities}\label{ss.p-adic-compatibilities}
We record now the compatibility of the local system $\xi \rho_{p,v}$ with $\mc{V}_{E_v}(\xi)$ for $\xi \in \Rep(G_{\QQ_p})$.  Let $D_\dR^\alg$ be the functor constructed in \cite{diao-lan-liu-zhu:log-riemann-hilbert}, from de~Rham $\QQ_p$-local systems on $S_{E_v}$ to the category of (algebraic) vector bundles on $S_{E_v}$ with a regular integrable connection and a Griffiths-transverse filtration. By the second paragraph of \cite[Theorem 5.3.1]{diao-lan-liu-zhu:log-riemann-hilbert},\footnote{Another application of superrigidity!} there is an isomorphism
\begin{equation}\label{eqn:DLLZ-compatibility}
    D_\dR^\alg(\xi \rho_{p,v}) \cong \mc{V}_{E_v}(\xi)
\end{equation}  
respecting the connections and the filtrations.

In fact, more is true, as explained in the third paragraph of \textit{loc.\ cit.}: $D_\dR^\alg(\xi\rho_{p,v})$ is the restriction to $S_{E_v}$ of a vector bundle $\bar{D}_\dR^\alg(\xi\rho_{p,v})$ on $\bar{S}_{E_v}$ with a connection having logarithmic singularities and nilpotent residues along $D_{E_v}$, and \textit{loc.\ cit.\ }constructs an isomorphism $\bar{D}_\dR^\alg(\xi\rho_{p,v}) \cong \bar{\mc{V}}_{E_v}(\xi)$ respecting connections and filtrations. (There is a unique isomorphism extending \cref{eqn:DLLZ-compatibility} since both logarithmic connections have nilpotent residues.) 

As a consequence, the analytification of the tensor functor $\mc{V}_{E_v}$ agrees with $D_\dR \circ \rho_{p, v}^{\ast, \an}$, where $D_\dR$ is the functor of \cite{liu-zhu:riemann-hilbert}. 

\subsection{Integral models}
We now construct our integral models following \cite[\S 8]{pila-shankar-tsimerman:andre-oort}. The toroidal compactification $\bar{S}$ is log smooth over $E$, and \emph{loc.\ cit.\ }shows that $(\bar{S}, S, D, s)$ spreads out to give $(\bar{\ms{S}}, \ms{S}, \ms{D}, s)$ over $\mc{O}_{E}[1/N]$ such that $\bar{\ms{S}}$ is proper over $\mc{O}_E[1/N]$, each of $\bar{\ms{S}}, \ms{S}$ the irreducible components of $\ms{D}$ are smooth over $\mc{O}_E[1/N]$ and moreover it is shown in \cite{pila-shankar-tsimerman:andre-oort} that the integral model can be chosen so that the following points hold:

\begin{assumption}\label{assumption:pst-int-model}
\, 
\begin{enumerate}
    \item The base change of $(\ms{S}, s)$ to $E$ is equal to $(S,s)$; 
    \item Any special point $y \in S(\bar{E})$ extends to an integral point $y \in \ms{S}(\mc{O}_{\bar{E}}[1/N])$;
    \item There is an embedding $\ms{S} \subseteq \bar{\ms{S}}$ of schemes over $\mc{O}_E[1/N]$ such that $\bar{\ms{S}}$ is smooth and proper over $\mc{O}_E[1/N]$ and $\ms{D} = \bar{\ms{S}} \setminus \ms{S}$ is a relative strict normal crossing divisor over $\mc{O}_E[1/N]$;
    %\item  $\Spec(\mc{O}_E[1/N]) \to \Spec(\ZZ[1/N])$ is smooth. 
\end{enumerate}
\end{assumption}

%The next lemma will slightly shrink our model (by enlarging $N$) so that it satisfies some necessary properties. Before the lemma, we observe that given any $M \in \MIC^\nilp(\bar{S}_\CC, D_\CC)$ is strongly cohomologically rigid \cite[Definition A.1]{pila-shankar-tsimerman:andre-oort} and admits a Griffiths-transverse filtration. For strong cohomological rigidity, use the Riemann--Hilbert correspondence to reduce to the corresponding statement for the corresponding Betti local system $\rho$, where we can apply Margulis superrigidity \cite{margulis} {(\color{red} Does anyone have a more specific reference?)} to ensure that $H^1(\pi_1^\top(S_\CC, \ol{s}), \rho) = 0$. The existence a of a Griffiths-transverse filtration is \cite[Theorem A.8]{pila-shankar-tsimerman:andre-oort}. 

\begin{lem}\label{lem:integral-model}
Fix a faithful representation $\xi_0 \in \Rep(G)$. After possibly replacing $N$ by some multiple, the integral model $(\bar{\ms{S}}, \ms{S}, \ms{D}, s)$ as above satisfies the following properties:
\begin{enumerate}
    \item For all primes $\ell$, $\rho_{K_0, s, \ell}^\ad$ factors through $\pi_1(\ms{S}[1/\ell], \bar{s})$; 

    %\item Let $r = \mr{rank}(\xi_0)$. Then for every $M \in \MIC^\nilp(\bar{S}_\CC, D_\CC)$ that is locally free of rank $r$, the vector bundle $M$ along with its connection and Griffiths-transverse filtration (defined above) spreads out to $\bar{\ms{S}}$. 
    \item The logarithmic flat vector bundle $\bar{\mc{V}}(\xi_0)\in\MIC(\bar{S}, D)$ (\cref{ss.canonical-g-bundles}) extends to a logarithmic flat vector bundle on $(\bar{\ms{S}}, \ms{D})$;
    
    \item For any finite prime $v$ of $\mc{O}_E[1/N]$ lying over a rational prime $p$ and any primes $\ell, \ell' \neq p$, the restrictions of $\rho_{K_0, s, \ell}^\ad$ and $\rho_{K_0, s \ell'}^\ad$ to $\ms{S}_v$ (the fiber of $\ms{S}$ over $v$) are companions. 
\end{enumerate}
\end{lem}
\begin{proof}
%Following the discussion preceeding the lemma, part (2) can be arranged immediately by spreading out as in \cite[Proposition A.10]{pila-shankar-tsimerman:andre-oort}. 
Part (2) is arranged by spreading out. Part (1) is proven in \cite[paragraph before Corollary 3.5]{patrikis-klevdal:shimura}.
Part (3) is \cite[Theorem 3.10]{patrikis-klevdal:shimura}. 
\end{proof}

\begin{rmk}
Since the monodromy groups of the canonical adjoint local systems are dense in the connected group $G^\ad$, the notion of strong and weak companions coincide. 
\end{rmk}

\subsection{Main theorem}
The main theorem of this paper extends the compatibility in \cref{lem:integral-model}-(3) to the case where one of $\ell, \ell'$ is equal to $p$.
First, we introduce some notation. For $v$ a finite place of $E$ with residue field $\kappa(v)$ of characteristic $p$, we write $\mc{O}_v$ for the completion of $\mc{O}_E$ at the maximal ideal $(v)$ and $E_v \coloneqq \mr{Frac}(\mc{O}_v) = \mc{O}_v[1/p]$. Let $S_{E_v}$ denote the base change of $S$ to $E_v$. If $p \nmid N$, let $\ms{S}_v$ be the fiber of $\ms{S}$ over $v$ (smooth over $\kappa(v)$), $\wh{\ms{S}}_v$ the $v$-adic completion of $\ms{S}$ (a smooth formal scheme over $\mr{Spf}(\mc{O}_v)$), $\wh{\ms{S}}_\eta$ the rigid generic fiber of $\wh{\ms{S}}_v$, and $S_{E_v}^\an$ the analytification of $S_{E_v}$.
Both $\wh{\ms{S}}_\eta$ and $S_{E_v}^\an$ are smooth rigid analytic varieties over $E_v$, and there is a natural open immersion $\wh{\ms{S}}_\eta \subseteq S_{E_v}^\an$, with equality if and only if $\ms{S}_{\mc{O}_v}$ is proper over $\mc{O}_v$. %The point $s \in \ms{S}(\mc{O}_{E}[1/N]$ gives rise to points in $S(E_v)$

For each prime $\ell$, we write $\rho_{\ell}=\rho_{K_0, s, \ell}$ and 
\begin{equation}
\qquad  \rho_{\ell, v} = \rho_{K_0, s, \ell}|_{\ms{S}_v}, \qquad \rho_{p,v} = \rho_{K_0, s, p}|_{S_{E_v}}.
\end{equation}
(here $\ell \neq p$). We write $\rho_{p,v}^\ast \colon \Rep(G_{\QQ_p}) \to \Loc_{\QQ_p}(S_{E_v}), \xi \mapsto \xi \circ \rho_{p,v}$, an exact $\QQ_p$-linear tensor functor. There is a natural analytification functor $\Loc_{\QQ_p}(S_{E_v}) \to \Loc_{\wh{\QQ}_p}(S_{E_v}^\an)$ (the latter category from \cite[\S 8]{scholze:p-adic-ht}). Composing $\rho_{p,v}^\ast$ with analytification (and restriction) gives exact $\QQ_p$-linear tensor functors
\begin{align}
\rho_{p,v}^{\an,\ast} \colon \Rep(G_{\QQ_p}) &\to \Loc_{\wh{\QQ}_p}(S_{E_v}^\an)  && \xi \mapsto (\xi \rho_{p,v})^\an, \\
\rho_{p,v}^{\circ,\ast} \colon \Rep(G_{\QQ_p}) &\to \Loc_{\wh{\QQ}_p}(\wh{\ms{S}}_{\eta}) && \xi \mapsto (\xi \rho_{p,v})^\an|_{\wh{\ms{S}}_\eta}. 
\end{align}

Let $\mc{V}^\circ_{E_v}$ be the composition
\begin{equation}
\Rep(G_{\QQ_p}) \xrightarrow{\mc{V}_{E_v}} \MIC(S_{E_v}) \to \MIC(S_{E_v}^\an) \to \MIC(\wh{\ms{S}}_\eta).
\end{equation}
This is an exact $\QQ_p$-linear tensor functor, which we call the canonical $G_{\QQ_p}$-bundle on $\wh{\ms{S}}_\eta$. 

For any of the objects above, we use the superscript $(\cdot)^\ad$ to denote the corresponding `adjoint quotient' (given by either postcomposing with $G \to G^\ad$ or precomposing with $\Rep(G^\ad) \to \Rep(G)$).

\begin{thm}\label{thm:overconvergent-Dcrys}
Continue with the notation above. There exists an integer $N'$ (with $N\mid N'$) such that for all $p \nmid N'$ and $v \mid p$, the following hold: 

\begin{enumerate}
\item {There is an exact $\QQ_p$-linear tensor functor 
    \[ \mc{E}_v^{\log} \colon \Rep(G_{\QQ_p}) \to \FIsoc(\bar{\ms{S}}_v, \ms{D}_v) \]
whose composition with $\FIsoc(\bar{\ms{S}}_v, \ms{D}_v) \to \MIC(\wh{\ms{S}}_\eta)$ is isomorphic to the canonical $G_{\QQ_p}$-bundle $\mc{V}_{E_v}^\circ$ on $\wh{\ms{S}}_\eta$.}
\end{enumerate}
{Let $\mc{E}_v^\dagger$ and $\mc{E}_v$ be the composition of $\mc{E}_v^{\log}$ with the restriction to $\FIsocdag(\ms{S}_v)$ and $\FIsoc(\ms{S}_v)$ respectively. }
\begin{enumerate}[resume]
\item {The $G(\QQ_p)$-local system $\rho_{p,v}^\circ$ is crystalline in the sense of \cite{guo-reinecke:prismatic} and associated to $\mc{E}_v$: for each representation $\xi \in \Rep(G_{\QQ_p})$, there is an isomorphism of sheaves on the pro-\'etale site of $\wh{\ms{S}}_\eta$ 
\begin{equation}\label{eqn:crys-comp}
c_{\crys}(\xi) \colon (\rho_{p,v}^{\circ,\ast}\xi) \otimes_{\hat{\QQ}_p} \mbb{B}_\crys \cong \mc{E}_v(\xi)(\mbb{B}_\crys) 
\end{equation}
which respects Frobenius and filtrations. Moreover, $c_{\crys}(\xi)$ is functorial in $\xi$ and compatible with tensor products.}

\item Let $\pi$ be a place of $\ol{\QQ}$ lying over $p$ and $\rho_{\pi, v}^\ad \colon \pi_1^{\FIsocdag}(\ms{S}_v) \to G^\ad_{\ol{\QQ}_\pi}$ the monodoromy representation of $\mc{E}_v^{\dagger, \ad} \otimes_{\QQ_p} {\ol{\QQ}_\pi}$. Then $\rho_{\pi, v}$ is surjective. 
\end{enumerate}
\end{thm}

The functor $\mc{E}^\dagger_v$ is the canonical overconvergent $G$-$F$-isocrystal on $\ms{S}_v$. By \cref{thm:overconvergent-Dcrys}-(3) we can apply \cref{thm:companions} to get a unique $\lambda$-adic companion of $\mc{E}^{\dagger, \ad}_v$; this allows us to state the main theorem of the paper:

\begin{thm}\label{thm:main}
Continue with notation as in \cref{thm:overconvergent-Dcrys}. Let $\lambda$ be a finite place of $\ol{\QQ}$ lying over a rational prime $\ell \neq p$. Let $\rho_{\pi \leadsto \lambda, v}^\ad$ be the $\lambda$-adic companion of $\mc{E}_v^{\dagger,\ad} \otimes_{\QQ_p} \ol{\QQ}_\pi$. Then the canonical $G^\ad(\QQ_\ell)$-adic local system $\rho_{\ell, v}^\ad$ is $G^\ad(\ol{\QQ}_\lambda)$-conjugate to $\rho_{\pi \leadsto \lambda, v}^\ad$. In particular, $\rho_{\ell, v}^\ad \otimes_{\QQ_\ell} \ol{\QQ}_\lambda$ is the $\lambda$-adic companion\footnote{In either the strong or weak sense, which are equivalent in this setting (see \cref{rmk:uniqueness-weak-companions}).} of $\mc{E}_v^{\dagger,\ad}$. 
\end{thm}

\begin{rmk}
\label{rmk:loss-of-primes}
At various points in both the construction of the smooth integral model $\ms{S} \to \Spec(\mc{O}_E[1/N])$ and the proofs of Theorems \ref{thm:overconvergent-Dcrys} and \ref{thm:main}, we have to exclude certain primes, which we now make explicit.   

There is an initial uncontrollable loss of primes in \cref{lem:integral-model} due to the fact that we only have access to the soft integral models obtained by spreading out. More specifically, we use the integral model of \cite[Theorem 8.1]{pila-shankar-tsimerman:andre-oort} (which assumes that $N$ is even) obtained by spreading out. We need to then enlarge $N$ to ensure that the canonical flat vector bundle $\mc{V}(\xi_0)$ spreads out to the integral model $\ms{S}$, and then to make a further enlargement of $N$ as in \cite{patrikis-klevdal:shimura} to ensure that for all $\ell$, the canonical $G^\ad(\QQ_\ell)$-local systems extend to $\ms{S}[1/\ell]$. The final enlargement $N'$ is as in \cite[Theorem 7.1]{pila-shankar-tsimerman:andre-oort} to get log-crystallinity of the canonical $p$-adic local system. We remark that the use of the Fontaine--Laffaille theory places some restriction on how small $N'$ can be; c.f.\ \cref{rmk:crystallinity-via-EG25}.
In particular, $N$ must be divisible by $2$ and all rational primes $p$ that ramify in $E$. 
It is seems possible that one could remove these restrictions by using only $F$-isocrystals; see \cref{rmk:crystallinity-w/o-FL-modules}. 
\end{rmk}

\section{Proof of the main theorem}\label{s.main-thm-proof}
\subsection{Proof of \cref{thm:overconvergent-Dcrys}-(1),(2)}
Parts (1) and (2) of \cref{thm:overconvergent-Dcrys} will follow from \cite[Theorem 7.1]{pila-shankar-tsimerman:andre-oort} (based on Esnault--Groechenig's work \cite[Theorem A.22]{pila-shankar-tsimerman:andre-oort}), which shows that $\rho_{p,v}^\ast\xi_0$ is log-crystalline in the sense that it arises as the \'etale realization of a logarithmic Fontaine--Laffaille module on $(\bar{\ms{S}}_v, \ms{D}_v)$ (c.f.\  \cref{ss.fontaine-laffaille}). Using the fully faithful \'etale realization functor 
	\[ T_\logcrys \colon \MF^{\nabla}_{[0,p-2], \mr{lf}}(\bar{\ms{S}}_v, \ms{D}_v) \to \Loc_{\ZZ_p}(S_{E_v}), \]
which is written down explicitly in \cite{liu-yang-zuo:log-crystalline} (where it is called $\mbb{D}^{\log}$) we record the log-crystallinity of $\rho_{p,v}^\ast\xi_0$:

\begin{thm}[{\cite[Theorem 7.1]{pila-shankar-tsimerman:andre-oort}}]
\label{thm:crystallinity}
Let $\xi_0$ be a faithful representation of $G$ as in \cref{lem:integral-model}.
Then there exists a multiple $N'$ of $N$ such that for all for all $p\nmid N'$ and any place $v \mid p$ of $E$, the local system $\rho_{p,v}^{\ast}\xi_0$ on $S_{E_v}$ is log-crystalline up to a cyclotomic twist, i.e.\ there exists a logarithmic Fontaine--Laffaille module $M \in \MF^{\nabla}_{[0,p-2], \mr{lf}}(\ol{\ms{S}}_{\mc{O}_v}, \ol{\ms{D}}_{\mc{O}_v})$ such that $T_\logcrys(M)(n)$, a Tate twist of the \'etale realization of $M$, is a $\ZZ_p$-lattice in $\rho_{p,v}^\ast \xi_0$.
\end{thm}

\begin{rmk}\label{rmk:crystallinity-via-EG25}
The engine in the proof of this theorem is, along with Margulis superrigidity, the result \cite[Theorem A.22]{pila-shankar-tsimerman:andre-oort} of Esnault--Groechenig which produces a ``log-crystalline descent" of the canonical $\CC$-local system on $S_\CC$ attached to $\xi_0$. It is an application of work of Lan--Sheng--Zuo (\cite{lan-sheng-zuo:higgs-de-Rham-flow}) and Lan--Sheng--Yang--Zuo (\cite{lan-sheng-yang-zuo:uniformization}) on Higgs--de Rham flows and Faltings (\cite{faltings:crys-cohom-p-adic-galois-reps}) on Fontaine--Laffaille modules.
We give a proof the crystallinity of $\rho_{p,v}$ that (rather than using \cite{lan-sheng-yang-zuo:uniformization, lan-sheng-zuo:higgs-de-Rham-flow}) is based on the more recent work of Esnault--Groechenig \cite{esnault-groechenig:cristallinity} which we find particularly elegant. (The overall structure of our argument follows that of Esnault--Groechenig \cite[Theorem A.22]{pila-shankar-tsimerman:andre-oort}, making adjustments to use \cite{esnault-groechenig:cristallinity} in place of \cite{lan-sheng-yang-zuo:uniformization, lan-sheng-zuo:higgs-de-Rham-flow}.) 
\end{rmk}

\begin{proof}[Proof of \cref{thm:crystallinity}]
Suppose $\xi_0$ has rank $n$. The main point is to find an $N'$ such that that for every $p\nmid N'$, the local system
    \[ \xi_0 \circ \rho_{p, \CC} \colon \pi_1(S_\CC, \ol{s}) \to K_{0,p} \xrightarrow{\xi_0} \GL_n(\bar{\ZZ}_p) \]
has \emph{some} log-crystalline descent. This means that there exists a power $q$ of $p$, a morphism $\Spec(W(\mbb{F}_q)) \to \Spec(\mc{O}_E[1/N'])$, an embedding $\iota \colon W(\mbb{F}_q) \to \CC$, and a representation $\Psi \colon \pi_1(S_{W(\mbb{F}_q[1/p]}, \ol{s}) \to \GL_r(\bar{\ZZ}_p)$ that (up to a Tate twist) lies in the essential image of $T_\logcrys$ and such that the restriction of $\Psi$ to $\pi_1(S_{\CC}, \ol{s})$ is isomorphic to $\xi_0 \circ \rho_{p, \CC}$.

Once the log-crystalline descent $\Psi$ is produced, the rest of the proof of \cite[Theorem 7.1]{pila-shankar-tsimerman:andre-oort} goes through to show that $\rho_{p,v}^\ast \xi_0$ is log-crystalline up to cyclotomic twist. (Note that the proof of \cite[Theorem 7.1]{pila-shankar-tsimerman:andre-oort} only explains the existence of  $T_\logcrys(M)(n)$ over some finite unramified extension of $E_v$, but this suffices because the relevant categories of Fontaine--Laffaille modules and \'etale local systems satisfy \'etale descent, cf.\ \cite[Remark after Theorem 2.6]{faltings:crys-cohom-p-adic-galois-reps}.) 

By writing $\xi_0$ as a direct sum of irreducible representations, we find that it is enough to prove the following:

\begin{claim}\label{claim:log-crystalline-descent}
Let $\rho^\top \colon \pi_1^\top(S_\CC, \ol{s}) \to \GL_r(\CC)$ be an irreducible representation with unipotent monodromy around $D_\CC$. Then 
\begin{enumerate}
    \item There exists a number field $F \subseteq \CC$ such that, after replacing $\rho^\top$ by a $\GL_r(\CC)$-conjugate, $\mr{Image}(\rho^\top) \subseteq \GL_r(\mc{O}_F)$. 
\end{enumerate}
Let $p$ be a prime that does not ramify in $E$ or $F$, and choose an an embedding $\mc{O}_F \to \breve{\ZZ}_p \coloneqq W(\bar{\mbb{F}}_p)$ and map $\Spec(W(k)) \to \Spec(\mc{O}_E[1/N'])$ for $k \subseteq \bar{\mbb{F}}_p$ a finite field. We define, for $L \coloneqq W(k)[p^{-1}]$ and $\ol{L}$ an algebraic closure of $L$, representations 
    \[ \rho_{\ol{L}} \colon \pi_1(S_{\ol{L}}, \ol{s}) \cong  \pi_1(S_\CC, \ol{s}) \to \GL_r(\breve{\ZZ}_p), \]
with the latter arrow the unique continuous extension of $\rho^\top$. Then 
\begin{enumerate}[resume]
    \item There a multiple $N'$ of $N$ such that for all $p \nmid N'$, there is (after possibly enlarging $k$) a representation 
        \[ \rho_{L} \colon \pi_1(S_L, \ol{s}) \to \GL_r(\breve{\ZZ}_p) \]
    that (up to a cyclotomic twist and scalar extension) lies in the essential image of\footnote{The category $\mbb{Z}_{p^f}\text{-}\MF_{[0, p-2], \lf}^\nabla$ consists of objects in $\MF_{[0, p-2], \lf}^\nabla$ equipped with a $\ZZ_{p^f}$-structure. }
        \[ T_\logcrys \colon \mbb{Z}_{p^f}\text{-}\MF_{[0, p-2], \lf}^\nabla(\bar{\ms{S}}_{W(k)}, \ms{D}_{W(k)}) \to \Loc_{\mbb{Z}_{p^f}}(S_{L}).\]
    for some $f$ and whose restriction to $\pi_1(S_{\ol{L}}, \ol{s})$ is $\rho_{\ol{L}}$. 
\end{enumerate}
and our aim is to show that each $\rho_{i, \ol{L}}$ is (the scalar extension of) the restriction to $\pi_1(S_{\ol{L}}, \ol{s})$ of the \'etale realization of some object of a Fontaine--Laffaille module (with $\ZZ_{p^f}$-structure for some $f \geq 1$). 
\end{claim} 

We will now prove the claim, where it is convenient to work with all irreducible $\CC$-local systems on $S_\CC$ at once. Our starting point is a consequence of Margulis superrigidity (\cite[Chapter IX, Theorem 6.15-(ii)]{margulis} or \cite[Chapter IX, Theorem 5.10-(ii)]{margulis}) which implies that every local system on $S_\CC$ is strongly cohomologically rigid. Hence for any fixed rank $r$, there are finitely many irreducible $\CC$-local systems on $S_\CC$ of rank $r$ with unipotent monodromy around the boundary divisor $D_\CC$. By rigidity \cite[Theorem 1.1]{esnault-groechenig:integrality}, we can choose integral representatives 
    \[\rho_1^\top, \ldots, \rho_n^\top \colon \pi_1^\top(S_\CC, \ol{s}) \to \GL_r(\mc{O}_F) \]
of these $\CC$-local systems, where $\mc{O}_F$ is the ring of integers in a number field $F$; this gives \cref{claim:log-crystalline-descent}-(1). For $p$ not ramified in $E$ or $F$, we fix $\mc{O}_F \to \breve{\ZZ}_p \coloneqq W(\bar{\mbb{F}}_p)$ and $\Spec(W(k)) \to \Spec(\mc{O}_E[1/N'])$ as in the claim, and we let
    \[ \rho_{i,\ol{L}} \colon \pi_1(S_{\ol{L}}, \ol{s}) \cong  \pi_1(S_\CC, \ol{s}) \to \GL_r(\breve{\ZZ}_p), \]
be the unique continuous extension of $\rho_i^\top$. Finally, we assume that $N'$ is a multiple of $N$ that satisfies the following:
\begin{enumerate}
    \item $N'$ is divisible by any of the rational primes that ramify in either $E$ or $F$ (we have already used this assumption to define $\rho_{i, \ol{L}}$).  
    \item Let $\bar{\mc{M}}_{1,\CC}, \ldots, \bar{\mc{M}}_{n,\CC} \in \MIC^\nilp(\bar{S}_\CC, D_\CC)$ correspond under the Riemann--Hilbert correspondence to $\rho_1^\top, \ldots, \rho_n^\top$. (We remark for later use that these represent all isomorphism classes of irreducible rank $R$ objects in $\MIC^\nilp(\bar{S}_\CC, D_\CC)$ by Riemann--Hilbert.) By Riemann--Hilbert and Margulis superrigidity \cite[Chapter IX, Theorem 6.15-(ii)]{margulis}, these are all strongly cohomologically rigid \cite[Definition A.1]{pila-shankar-tsimerman:andre-oort} and hence each $\bar{\mc{M}}_{i,\CC}$ admits a Griffiths-transverse filtration $\Fil^\bullet \mc{M}_{i, \CC}$ \cite[Theorem A.8]{pila-shankar-tsimerman:andre-oort}. Then we assume that each of these vector bundles $\bar{\mc{M}}_{i, \CC}$ with connection and filtration spreads out to give a strongly cohomologically rigid vector bundle $\bar{\mc{M}}_{i, \mc{O}_E[1/N']}$ with filtration and connection (relative to $\mc{O}_E[1/N']$).
    \item $N'$ is divisible by any prime $p$ where at least one of the filtrations considered above on some $\bar{\mc{M}}_{i,\mc{O}_{E[1/N']}}$ is \emph{not} concentrated in degree $[0, p-2]$.\footnote{This final condition is needed so that we can apply the global Fontaine--Laffaille theory.}

\end{enumerate}

We let $\bar{\mc{M}}_{i, W(k)}, \bar{\mc{M}}_{i,L}$ be the pullback of $\bar{\mc{M}}_{i, \mc{O}_E[1/N']}$ to $\Spec(W(k)), \Spec(L)$ respectively. Since $\bar{\mc{M}}_{i, W(k)}$ have a $W(k)$-linear connection and filtration, it only remains to put Frobenius structure on each to produce Fontaine--Laffaille modules. 
%Using topological properties of Frobenius pullback \cite[Corollary 5.5, Remark 5.6]{esnault-groechenig:cristallinity}, Esnault--Groechenig show that each $\bar{\mc{M}}_{i,W(k)}$ is $f$-periodic in the sense of \cite[Definition 3.37]{esnault-groechenig:cristallinity} and produces an $f$-periodic Higgs--de Rham flow (c.f.\ \cref{rmk:HdR-flow-explicit}). 
The Frobenius pullback functor $F^\ast$ acts as a permutation on the set of isomorphism classes of $\bar{\mc{M}}_{1, L}, \ldots, \bar{\mc{M}}_{n, L}$ 
\cite[Lemma 4.4, Lemma 5.4]{esnault-groechenig:cristallinity}, so by the pigeonhole principle we have $(F^\ast)^f \mc{M}_{i,L} \cong \mc{M}_{i, L}$ for some $f \geq 1$. This produces the $F^f$-isocrystals that underlie the Fontaine--Laffaille modules. One then produces the $f$-periodic Higgs--de Rham flows using \cite[Corollary 3.40]{esnault-groechenig:cristallinity} and constructs $\mathbb{Z}_{p^f}$-Fontaine--Laffaille modules as in \cref{rmk:HdR-flow-explicit} (here we need $k$ to contain $\mbb{F}_{p^f}$). Applying $T_\logcrys$ to the $\ZZ_{p^f}$-Fontaine--Laffaille modules associated to the $\bar{\mc{M}}_{i, W(k)}$ gives a collection of crystalline representations
    \[ \rho_{i, L}' \colon \pi_1(S_L, \ol{s}) \to \GL_r(\ZZ_{p^f}). \]
Finally, defining $\rho_{i, \ol{L}}' \coloneqq \rho_{i, L}'|\pi_1(S_{\ol{L}}, \ol{s})$, we have an inclusion
    \[ \{ [\rho_{i, \ol{L}}'] : i = 1, \ldots, n \} \subseteq \{ [\rho_{i, \ol{L}}]: i = 1,\ldots, n \}, \]
where the brackets represent $\GL_r(\ol{\QQ}_p)$-conjugacy classes. Therefore, the pigeonhole principle will finish the proof of the claim once we show that the $[\rho_{i, \ol{L}}']$ are all distinct.

Assume that $[\rho_{i, \ol{L}}'] = [\rho_{j, \ol{L}}']$.
Then there is some finite extension $M$ of $L$ such that the restictions $\rho_{i,M}', \rho_{j, M}'$ of $\rho_{i,L}', \rho_{j,L}'$ to $\pi_1(S_F,\ol{s})$ are isomorphic. This implies that $D_\dR^{\log}(\rho_{i,M}') = D_\dR^{\log}(\rho_{i, L}')|_{S_M}$ and $D_\dR^{\log}(\rho_{j,F}') = D_\dR^{\log}(\rho_{j, L}')|_{S_M}$ are isomorphic, for $D_\dR^{\log}$ as in \cite[Theorem 3.27]{diao-lan-liu-zhu:log-riemann-hilbert}.\footnote{Strictly speaking, we need to input a Kummer \'etale $\ZZ_p$-local system, but the construction of $T_\logcrys(\bar{\mc{M}}_{i,W(k)})$ in \cite{liu-yang-zuo:log-crystalline} is visibly the restriction of a Kummer \'etale $\ZZ_p$-local system.} But \cite[Theorem 2.9]{pila-shankar-tsimerman:andre-oort} shows that $D_\dR^{\log}(\rho_{i, L}') = \bar{\mc{M}}_{i, L}$ and $D_\dR^{\log}(\rho_{j, L}') = \bar{\mc{M}}_{j, L}$. 
Therefore, $\bar{\mc M}_{i,\mathbb C}=\bar{\mc M}_{j,\mathbb C}$, so $i=j$, as desired.
\end{proof}

\begin{rmk}\label{rmk:crystallinity-w/o-FL-modules}
It seems possible that the argument for crystallinity of the canonical local systems could be somewhat simplified to use only $F$-isocrystal theory and bypass the more delicate Fontaine--Laffaille theory: in the above argument we constructed directly the $F^f$-isocrystals $\bar{\mc{M}}_{i, L}$ first, and then constructed Fontaine--Laffaille modules to produce crystalline local systems. If one had an equivalence of categories between log crystalline local systems and logarithmic $F$-isocrystals (e.g.\ as in \cite[Theorem 3.18]{tan-tong:crystalline} in the non logarithmic setting; it is possible that this follows from results in \cite{guo-yang:pointwise}) then one can produce the requisite crystalline local systems on $S_L$ directly from the $F^f$-isocrystals, and make the pigeonhole argument using compatibility of the appropriate $\mbb{D}_\crys$,  $D_\dR$ and $D_\dR^{\alg}$. We do not pursue this line of reasoning here. 
\end{rmk}

We now return to the proof of \cref{thm:main}-(1) with the knowledge that $\rho_{p, v}^\ast(\xi_0)$ contains a $\ZZ_p$-lattice isomorphic to $T_\logcrys(M)(n)$ for some $M \in \MF^{\nabla}_{[0,p-2], \mr{lf}}(\bar{\ms{S}}_{\mc{O}_v}, \ms{D}_{\mc{O}_v})$. Define $\mc{E}_v^{\log}(\xi_0) \in \FIsoc(\bar{\ms{S}}_v, \ms{D}_v)$ to be the logarithmic $F$-isocrystal associated to $M(n)$ (\cref{rmk:MF-to-FIsoc}), and $\mc{E}^\dagger(\xi_0), \mc{E}_v(\xi_0)$ the restriction to $\FIsocdag(\ms{S}_v), \FIsoc(\ms{S}_v)$ respectively. Then \cite[Remark after Theorem 5.6]{faltings:crys-cohom-p-adic-galois-reps} shows that $\mc{E}_v(\xi_0)$ is \emph{associated} to $\rho_{p,v}^{\circ, \ast} \xi_0$ in the sense of Faltings, c.f.\ the discussion preceding \cite[Lemma 5.5]{faltings:crys-cohom-p-adic-galois-reps}, and hence $\rho_{p,v}^{\circ, \ast}(\xi_0)$ is crystalline in the sense of Faltings. We remark that this is \emph{precisely} the definition of crystallinity used in \cite[Definition 2.31]{guo-reinecke:prismatic} and is equivalent to the definition of crystallinity in \cite{tan-tong:crystalline} by \cite[Proposition 3.21]{tan-tong:crystalline}; when we refer to crystalline local systems in the remainder of the section, we mean one of these equivalent definitions.  
%In particular, the sheaf $\rho_{p,v}^{\circ, \ast} \xi_0$ is crystalline, in the sense of \cite[Definition 2.31]{guo-reinecke:prismatic} or \cite[Definition 3.10]{tan-tong:crystalline} which are equivalent (see also \cite{brinon:relatif} for the local theory).

We let $\Loc_{\hat{\QQ}_p}^\crys(\wh{\ms{S}}_\eta)$ denote the full subcategory of $\Loc_{\hat{\QQ}_p}(\wh{\ms{S}}_\eta)$ consisting of crystalline local systems. This is a Tannakian subcategory, meaning in particular that it is closed under subquotients.\footnote{
This can be seen, for example, by applying the pointwise criterion of crystallinity \cite{guo-yang:pointwise} to reduce to the case of Galois representations, where stability under subquotients is well-known.
%Alternatively, it can be deduced from \cite[Corollary 3.16]{tan-tong:crystalline} and \cite[Theorem 8.4.2]{brinon:relatif} that the category of crystalline local systems is stable under subobjects, which suffices for us by the semisimplicity of $\Rep(G_{\QQ_p})$.
}
Therefore, since $\xi_0$ is a tensor generator of $\Rep(G_{\QQ_p})$, the image of $\rho_{p,v}^{\circ,\ast}$ lies in the subcategory of crystalline local systems.
So we may define $\mc E_v$ to be the following composition of functors:
% https://q.uiver.app/#q=WzAsMyxbMCwwLCJcXFJlcChHX3tcXFFRX3B9XlxcYWQpIl0sWzEsMCwiXFxMb2Nfe1xcaGF0e1xcUVF9X3B9XlxcY3J5cyhcXHdoe1xcbXN7U319X3ZeXFxhbikiXSxbMSwxLCJcXEZJc29jKFxcbXN7U31fdikiXSxbMSwyLCJcXG1iYntEfV9cXGNyeXMiXSxbMCwyLCJcXG1je0V9IiwyXSxbMCwxLCJcXHJob197cCx2fV57XFxjaXJjLCBcXGFzdH0iXV0=
\[\begin{tikzcd}
	{\Rep(G_{\QQ_p})} & {\Loc_{\hat{\QQ}_p}^\crys(\wh{\ms{S}}_\eta)} \\
	& {\FIsoc(\ms{S}_v)}.
	\arrow["{\rho_{p,v}^{\circ, \ast}}", from=1-1, to=1-2]
	\arrow["{\mc{E}_v}"', from=1-1, to=2-2]
	\arrow["{\mbb{D}_\crys}", from=1-2, to=2-2]
\end{tikzcd}\]
Here $\mbb{D}_\crys$ is defined in \cite[Definition 3.12]{tan-tong:crystalline}. % and is an equivalence of Tannakian categories \cite[Remark 3.19]{tan-tong:crystalline}. 
Note also by \cite[Proposition 3.21]{tan-tong:crystalline}, there is no ambiguity in the meaning of $\mc{E}_v(\xi_0)$, and more importantly $\mc{E}_v$ satisfies the conclusion of  \cref{thm:overconvergent-Dcrys}-(2). To finish the proof of \cref{thm:overconvergent-Dcrys}-(1),(2), it remains to extend $\mc{E}_v$ to a tensor functor $\mc{E}_v^{\mr{log}}$. For this, consider the following diagram: 
% https://q.uiver.app/#q=WzAsOCxbMiwxLCJcXEZJc29jZGFnKFxcbXN7U31fdikiXSxbMSwxLCJcXEZJc29jKFxcYmFye1xcbXN7U319X3YsIFxcbXN7RH1fdikiXSxbNCwxLCJcXEZJc29jKFxcbXN7U31fdikiXSxbMiwyLCJcXElzb2NkYWdfRihcXG1ze1N9X3YpIl0sWzQsMiwiXFxJc29jKFxcbXN7U31fdikiXSxbMywzLCJcXE1JQyhcXHdoe1xcbXN7U319X1xcZXRhKSJdLFsxLDMsIlxcTUlDXntcXG1ye25pbHB9fShcXGJhcntTfV97RV92fSwgRF97RV92fSkiXSxbMCwwLCJcXFJlcChHX3tcXFFRX3B9KSJdLFs2LDVdLFszLDUsIlxcc21hbGwgKFxcbXJ7QX0nKSIsMV0sWzQsNSwiXFxzbWFsbCAoXFxtcntBfScnJykiLDFdLFszLDQsIlxcc21hbGwgKFxcbXJ7QX0nJykiLDFdLFswLDJdLFswLDMsIlxcc21hbGwgKFxcbXJ7Qn0nKSIsMV0sWzIsNF0sWzEsNl0sWzEsMCwiXFxzbWFsbCAoXFxtcntDfScpIiwxXSxbNywyLCJcXG1iYntEfV9cXGNyeXMgXFxjaXJjIFxccmhvX3twLHZ9XntcXGNpcmMgXFxhc3R9IiwwLHsib2Zmc2V0IjotMSwiY3VydmUiOi0yfV0sWzcsNiwiXFxiYXJ7XFxtY3tWfX1fe0Vfdn0iLDIseyJjdXJ2ZSI6M31dLFs3LDMsIlxcc21hbGwgKFxcbXJ7QX0pIiwxLHsib2Zmc2V0IjoxLCJjdXJ2ZSI6NSwic3R5bGUiOnsiYm9keSI6eyJuYW1lIjoiZGFzaGVkIn19fV0sWzcsMCwiXFxzbWFsbCAoXFxtcntCfSkiLDEseyJjdXJ2ZSI6LTEsInN0eWxlIjp7ImJvZHkiOnsibmFtZSI6ImRhc2hlZCJ9fX1dLFs3LDEsIlxcc21hbGwgKFxcbXJ7Q30pIiwxLHsic3R5bGUiOnsiYm9keSI6eyJuYW1lIjoiZGFzaGVkIn19fV1d
\[\begin{tikzcd}
	{\Rep(G_{\QQ_p})} &&&& \\
	& {\FIsoc(\bar{\ms{S}}_v, \ms{D}_v)} & {\FIsocdag(\ms{S}_v)} && {\FIsoc(\ms{S}_v)} \\
	&& {\IsocdagF(\ms{S}_v)} && {\Isoc(\ms{S}_v)} \\
	& {\MIC^{\mr{nilp}}(\bar{S}_{E_v}, D_{E_v})} && {\MIC(\wh{\ms{S}}_\eta)}
	\arrow["{\small (\mr{C})}"{description}, dashed, from=1-1, to=2-2]
	\arrow["{\small (\mr{B})}"{description}, curve={height=-6pt}, dashed, from=1-1, to=2-3]
	\arrow["{\mbb{D}_\crys \circ \rho_{p,v}^{\circ \ast}}", shift left, curve={height=-12pt}, from=1-1, to=2-5]
	\arrow["{\small (\mr{A})}"{description}, shift right, curve={height=30pt}, dashed, from=1-1, to=3-3]
	\arrow["{\bar{\mc{V}}_{E_v}}"', curve={height=18pt}, from=1-1, to=4-2]
	\arrow["{\small (\mr{C}')}"{description}, from=2-2, to=2-3]
	\arrow[from=2-2, to=4-2]
	\arrow[from=2-3, to=2-5]
	\arrow["{\small (\mr{B}')}"{description}, from=2-3, to=3-3]
	\arrow[from=2-5, to=3-5]
	\arrow["{\small (\mr{A}'')}"{description}, from=3-3, to=3-5]
	\arrow["{\small (\mr{A}')}"{description}, from=3-3, to=4-4]
	\arrow["{\small (\mr{A}''')}"{description}, from=3-5, to=4-4]
	\arrow[from=4-2, to=4-4]
\end{tikzcd}\]
All solid straight lines are the obvious restriction or forgetful functors. The functor $\bar{\mc{V}}_{E_v}$ is defined in \cref{ss.canonical-g-bundles}, and the $2$-commutativity\footnote{i.e.\ commutativity up to natural isomorphism} of the outer diagram follows from the discussion of \cref{ss.p-adic-compatibilities} and compatibility of $\mbb{D}_\crys$ and $D_\dR$ \cite[Proposition 3.22]{tan-tong:crystalline}. %Here $D_\dR^{\log}$ is the functor of \cite[Theorem 3.2.12]{diao-lan-liu-zhu:log-riemann-hilbert} ({\color{red}Check the numbering! I only have access to the arXiv version of this paper}) where we use that . \addref (the composition $D_\dR^{\log}\circ \rho_{p,v}$ is non other than the tensor functor $\bar{V}$ of \cref{smth}) and the outer the outer diagram 2-commutes\footnote{I.e. commutes up to natural isomorphism.} by compatibility of $\mbb{D}_{\crys}$ and $D_{\dR}^{\log}$ with $D_\dR$  (\addref and \addref respectively). 
We claim that we can define the dashed arrows so that the entire diagram is 2-commutative. 

We begin with the functor $(\mr{A})$. First note that $(\mr{A}')$ is fully faithful, since both $(\mr{A}'')$ \cite[Corollary 5.7]{daddezio-esnault:universal-ext} and $(\mr{A}''')$ \cite[Theorem 2.15]{ogus:f-isoc-2} are. Thus we can view $\IsocdagF(\ms{S}_v)$ as a full subcategory of $\MIC(\wh{\ms{S}}_\eta)$, and we just need to check that for any $\theta \in \Rep(G_{\QQ_p})$, the underlying integrable vector bundle of $\mbb{D}_\crys \circ \rho_{p,v}^{\circ \ast}(\theta)$ is in the essential image of $(\mr{A}')$. This holds for $\xi_0$ since $\mbb{D}_\crys \circ \rho_{p,v}^{\circ \ast}(\xi_0) = \mc{E}_v(\xi_0)$ has a logarithmic extension $\mc{E}_v^{\log}(\xi_0)$. Consequently it holds for any direct sum of tensor powers of $\xi_0$ and its dual. Since $G_{\QQ_p}$ is reductive and $\xi_0$ is faithful, any other $\theta \in \Rep(G_{\QQ_p})$ is a direct summand of such a representation, and we are done once we show that the essential image of $(\mr{A}')$ is closed under direct summands. This follows since $(\mr{A}')$ is fully faithful: if $(\mr{A}')(M) \in \MIC(\wh{\ms{S}}_v)$ lies in the essential image and $N$ is a direct summand of $(\mr{A}')(M)$ in $\MIC(\wh{\ms{S}}_v)$, then the projector $(\mr{A}')(M) \twoheadrightarrow N \hookrightarrow (\mr{A}')(M)$ arises from a unique projector $M \to M$ in $\IsocdagF(\ms{S}_v)$ whose image maps to $N$ under $(\mr{A}')$. 

In order to define the functor $(\mr{B})$ we use again that $(\mr{A}'')$ is fully faithful, so we can use the $F$-structure on $\mbb{D}_\crys \circ \rho_{p, v}^{\circ \ast}$ to lift $(\mr{A})$ uniquely to a functor $(\mr{B})$ (making that part of the diagram $2$-commute). Finally, to define the functor $(\mr{C})$, we use a similar argument based on the fact that $(\mr{C}')$ is fully faithful \cite[Theorem 6.4.5]{kedlaya:semistable-reduction}. The tensor functor 
    \[ \mc{E}_v^{\log} \colon \Rep(G_{\QQ_p}) \to \FIsoc(\bar{\ms{S}}_v, \ms{D}_v) \]
of \cref{thm:overconvergent-Dcrys}-(1) is then functor $(\mr{C})$ in the diagram above. This finishes the proof. 

\begin{rmk}
Our approach to proving the overconvergence of $\mc{E}_v$ is ad hoc but well-suited for our purposes. A more systematic approach (using general principles of relative $p$-adic Hodge theory) should allow one to prove a more general theorem on overconvergence of $\mbb{D}_\crys$.
Namely, suppose we are in the setting where $X$ is a smooth separated scheme over $W = W(\kappa)$ for $\kappa$ a perfect field of characteristic $p$ with special fiber $X_\kappa$, $p$-adic completion $\wh{X}$, rigid generic fiber $\wh{X}_\eta$, and generic fiber $X_K$ with analytification $X_K^\an$.
Then $\wh{X}_\eta \subseteq X_K^\an$ is an open immersion, and we have a relative $\mbb{D}_\crys$ functor 
    \[ \mbb{D}_\crys \colon \Loc_{\QQ_p}(\wh{X}_\eta) \to \FIsoc(X_\kappa), \]
with $\FIsoc(X_\kappa)$ the category of convergent $F$-isocrystals. A general crystalline local system $\mbb{V} \in \Loc_{\QQ_p}(\wh{X}_\eta)$ will not admit an overconvergent extension of $\mbb{D}_{\crys}(\mbb{V})$ (as the theory of the canonical subgroup on the ordinary-reduction locus of the modular curve shows), but if $\mbb{V}$ admits an extension $\tilde{\mbb{V}}$ to a $\QQ_p$-local system on some strict neighborhood $U$ with $\wh{X}_\eta \subsetneq U \subseteq X_K^\an$, then we expect to have an overconvergent extension $\mbb{D}^\dagger_\crys(\mbb{V})$ of $\mbb{D}_\crys(\mbb{V})$. One should be able to argue, assuming existence of appropriate $\mbb{D}_\crys$ and $\mbb{D}_{\mr{st}}$ functors, as follows: by \cite{liu-zhu:riemann-hilbert}, $\tilde{\mbb{V}}$ is de~Rham and hence potentially log-crystalline (by combining results of \cite{kedlaya:monodromy-in-families} and \cite{guo-yang:pointwise}). The overconvergence of $\mbb{D}_\crys({\mbb{V}})$ can be checked \'etale locally, so we can thus assume $\tilde{\mbb{V}}$ is log-crystalline, and the convergence of $\mbb{D}_\mr{st}(\tilde{\mbb{V}}) = \mbb{D}_\crys(\tilde{\mbb{V}})$ gives the desired overconvergence. It seems possible that the $\mbb{D}_{\mr{st}}$ functor of \cite[Theorem 5.20]{guo-yang:pointwise} has the necessary properties, though the authors have not thought about this (one potential subtlety is that \emph{loc.\ cit.\ }requires smooth formal models, while moving beyond the rigid generic fiber $\wh{X}_\eta$ is only possible with semistable formal models). 
%Nevertheless, our proof of the main theorems of this paper (in particular, the compatibility part) require the finer knowledge about $\mc{E}_v^\dagger$ coming from \cite{diao-lan-liu-zhu:log-riemann-hilbert} and \cite{esnault-groechenig:cristallinity}.
\end{rmk}

\subsection{Proof of \cref{thm:overconvergent-Dcrys}-(3)}
{For brevity, let us say that a functor ``reflects subobjects/direct summands'' to mean that its essential image is closed under taking subobjects/direct summands.} We observed earlier that $\FIsoc(\bar{\ms{S}}_v, \ms{D}_v) \to \FIsocdag(\ms{S}_v)$ is fully faithful; it reflects subobjects by \cite{kedlaya:semistable-reduction} (see also \cite{shiho:log-ext-overconv-isoc}). Indeed, the essential image consists of the overconvergent $F$-isocrystals which have unipotent monodromy \cite[Definition 4.4.2]{kedlaya:semistable-reduction}, and for any overconvergent $F$-isocrystal with unipotent monodromy, any subobject also has unipotent monodromy (this reduces to a local statement, which is shown in the proof of \cite[Proposition 3.2.20]{kedlaya:semistable-reduction}).

By \cite[Proposition 2.21]{deligne-milne}, proving \cref{thm:overconvergent-Dcrys}-(3) is equivalent to showing that the composition (using the notion of scalar extension of a Tannakian categories/exact tensor functors in \cite{deligne-milne})
    \[ \mc{E}_{v, \ol{\QQ}_\pi}^{\dagger, \ad} = \mc{E}_v^{\dagger, \ad} \otimes_{\QQ_p} \ol{\QQ}_\pi \colon \Rep(G^\ad_{\ol{\QQ}_\pi}) \to \Rep(G_{\ol{\QQ}_\pi}) \to \FIsocdag(\bar{\ms{S}}_v, \ms{D}_v)_{\ol{\QQ}_\pi} \]
is fully faithful and reflects subobjects. Since the functor $\FIsoc(\bar{\ms{S}}_v, \ms{D}_v) \to \FIsocdag(\ms{S}_v)$ is fully faithful and reflects subobjects, we need to show that 
    \[ \mc{E}_{v, \ol{\QQ}_\pi}^{\log,\ad} = \mc{E}_v^{\log,\ad} \otimes_{\QQ_p} \ol{\QQ}_\pi \colon \Rep(G^\ad_{\ol{\QQ}_\pi}) \to \Rep(G_{\ol{\QQ}_\pi}) \to \FIsoc(\bar{\ms{S}}_v, \ms{D}_v)_{\ol{\QQ}_\pi} \]
is fully faithful and reflects subobjects. Since $\Rep(G^\ad_{\ol{\QQ}_\pi})$ is semisimple, every subobject is a direct summand, so it suffices to show that $\mc{E}_{v, \ol{\QQ}_\pi}^{\log, \ad}$ is fully faithful (by the same argument previously made that fully faithful functors reflect direct summands).  

To do so, fix an isomorphism $\iota \colon \ol{\QQ}_\pi \xrightarrow{\sim} \CC$, and consider the diagram
% https://q.uiver.app/#q=WzAsMyxbMCwwLCJcXFJlcChHXlxcYWRfe1xcUVFfcH0pIl0sWzEsMCwiXFxGSXNvY15cXG5pbHAoXFxiYXJ7XFxtc3tTfX1fdiwgXFxtc3tEfV92KV97XFxvbHtcXFFRfV9cXHBpfSAiXSxbMSwxLCJcXFJlcF9cXENDKFxcR2FtbWEpIl0sWzAsMiwiQiIsMV0sWzEsMiwiQiciXSxbMCwxLCJcXG1je0V9X3t2LCBcXG9se1xcUVF9X1xccGl9XntcXGxvZ30iXV0=
\begin{equation}\label{eqn:monodromy-triangle}
\begin{tikzcd}
	{\Rep(G^\ad_{\ol{\QQ}_\pi})} & {\FIsoc(\bar{\ms{S}}_v, \ms{D}_v)_{\ol{\QQ}_\pi} } \\
	& {\Rep_\CC(\Gamma)}
	\arrow["{\mc{E}_{v, \ol{\QQ}_\pi}^{\log,\ad}}", from=1-1, to=1-2]
	\arrow["B"', from=1-1, to=2-2]
	\arrow["{B'}", from=1-2, to=2-2]
\end{tikzcd}
\end{equation}

where $B$ is the composition $\Rep(G^\ad_{\ol{\QQ}_\pi}) \xrightarrow{\sim} \Rep(G^\ad_\CC)\to\Rep_\CC(\Gamma)$ induced by $\iota$ and pullback along $\rho^\mr{top} \colon \Gamma = \pi_1^{\mr{top}}(S_\CC, \bar{x}) \to G^\ad(\CC)$ and $B'$ is the composition
\[
\begin{tikzcd}
\FIsoc(\bar{\ms{S}}_v, \ms{D}_v)_{\ol{\QQ}_\pi} 
\arrow[r]&
\MIC^{\nilp}(\bar{S}_{E_v}^\an, D_{E_v}^\an)_{\ol{\QQ}_\pi}
\arrow[r,"\mr{GAGA}"]&
\MIC^{\nilp}(\bar{S}_{E_v}, D_{E_v})_{\ol{\QQ}_\pi}
\arrow[lld,"\iota"',out=-120,in=60,looseness=0.2]\\
\MIC^{\nilp}(\bar{S}_\CC, D_\CC)
\arrow[r,"\mr{GAGA}"]&
\MIC^{\nilp}(\bar{S}_\CC^\an, D_\CC^\an)
\arrow[r,"\mr{restrict}"]&
\MIC(S_\CC^\an)
\arrow[r,"\mr{RH}"]&
\Rep_\CC(\Gamma)
\end{tikzcd}
\]
(here we used rigid GAGA, complex GAGA, and the complex-analytic Riemann--Hilbert correspondence).
The construction of $\mc{E}^{\log}_v$ identifies the composition of $\mc{E}_{v}^{\log,\ad}$ and the arrows in the top row with the canonical logarithmic $G$-bundle on $\bar{S}_{E_v}$ of \cref{ss.canonical-g-bundles}. We may thus apply \cite[Theorem 5.3.1]{diao-lan-liu-zhu:log-riemann-hilbert} to conclude that the triangle \eqref{eqn:monodromy-triangle} is 2-commutative. 
Now $B$ is fully faithful since $\Gamma$ is Zariski-dense in $G^\ad(\CC)$, and $B'$ is faithful, so $\mc{E}_{v, \ol{\QQ}_\pi}^{\log,\ad}$ is fully faithful.

\subsection{Proof of \cref{thm:main}}
The following lemma immediately implies \cref{thm:main} upon applying it with $(\rho_1,\rho_2)=(\rho_{\pi\leadsto\lambda,v}^\ad,\rho_{\ell,v}^\ad)$; the necessary hypotheses are checked afterwards. 
\begin{lem}
Let $\lambda$ be a finite place of $\ol{\QQ}$ and $\rho_1,\rho_2\colon\pi_1(\ms{S}_{v})\to G^\ad(\ol{\QQ}_\lambda)$ be two tame $G^\ad(\ol{\QQ}_\lambda)$-local systems with Zariski-dense image. Then 
\begin{enumerate}
\item There exists $\tau\in\Aut(G^\ad_{\ol{\QQ}_\lambda})$ such that $\rho_1 = \tau \circ \rho_2$.
\end{enumerate}
Assume further that $\rho_1, \rho_2$ are compatible at special points in the following sense: for every special point $s \in \ms{S}(\mc{O}_C[1/N])$ (for $C$ a finite extension of $E$) and every closed point $x \in s \cap |\ms{S}_v|$  the semisimple conjugacy classes $[\rho_i(\Frob_x)], i = 1,2$ in $[G^\ad\sslash G^\ad](\ol{\QQ}_\lambda)$ are equal. Then
\begin{enumerate}[resume]
\item The automorphism $\tau$ of $(1)$ is inner. In particular, the semisimple Frobenius conjugacy classes $[\rho_i(\Frob_x)]$ are equal for all closed points $x \in |\ms{S}_v|$. 
\end{enumerate}
\end{lem}
\begin{proof}
(1) This is essentially \cite[Proposition 2.3-(2)]{patrikis-klevdal:shimura}; we recall the proof. 
Let $\ol{\rho}_i$ be the restriction of $\rho_i$ to the geometric fundamental group $\pi_1(\ms{S}_{\ol{v}})$ (for $\ms{S}_{\ol{v}}$ the base change of $\ms{S}_v$ to $\ol{\kappa}(v)$), and let $H \subseteq G^\ad_{\ol{\QQ}_\lambda}$ be the Zariski-closure of the image of $\ol{\rho}_i$.
Since $\pi_1(\ms{S}_{\ol{v}})$ is a normal subgroup of $\pi_1(\ms{S}_v)$ with quotient $\Gal(\ol{\kappa(v)}/\kappa(v)) \cong \wh{\ZZ}$, and $\rho_i$ has Zariski-dense image, $H$ is also normal in $G^\ad_{\ol{\QQ}_\lambda}$ with commutative quotient.
But $G^\ad_{\ol{\QQ}_\lambda}$ is the product of its simple factors, so $H=G^\ad_{\ol{\QQ}_\lambda}$.

Consider now the homomorphism $\rho_i^\mr{top}$ defined as the composition
\[
\Gamma
\to
\pi_1(S_\CC) \xrightarrow{\sim}
\pi_1(S_{\ol{\QQ}})
\xrightarrow{\mr{sp}}
\pi_1^\mr{t}(\ms{S}_{\ol{v}})
\xrightarrow{\ol{\rho}_i}
G^\ad(\ol{\QQ}_\lambda),
\]
where $\mr{sp}$ is the specialization homomorphism (see e.g.\ Corollary 2.6 in the first arXiv version of \cite{patrikis-klevdal:shimura}).
Since $\mr{sp}$ is surjective, it follows from the previous paragraph that $\rho_i^\mr{top}$ has Zariski-dense image.
By a version of Margulis's superrigidity theorem (see \cite[Proposition 3.3]{patrikis-klevdal:shimura}, which applies because of \cref{assumption:superrigid}!), we have $\rho_1^\mr{top}=\tau \circ \rho_2^\mr{top}$ for some $\tau \in \Aut(G^\ad_{\ol{\QQ}_\lambda})$.
The image of $\Gamma$ is dense in $\pi_1(S_{\CC})$, so $\ol{\rho_1}=\tau \circ \ol{\rho_2}$. Thus $\rho_1$ and $\tau\circ \rho_2$ are both extensions of $\ol{\rho_1}=\tau \circ \ol{\rho_2}$ to $\pi_1(\ms{S}_v)$, and hence are equal by \cite[Lemma 2.1]{patrikis-klevdal:shimura}. 

(2).
This follows from the same argument as given for \cite[Theorem 3.10]{patrikis-klevdal:shimura}. More specifically, we construct special points using \cite[Proposition 3.7, Lemma 3.9]{patrikis-klevdal:shimura} to apply \cite[Proposition 2.8]{patrikis-klevdal:shimura}. 
\end{proof}

It is shown in the course of the proof of \cite[Proposition 2.3]{patrikis-klevdal:shimura} that the representation $\rho_{\lambda, v}^\ad$ has Zariski-dense image in $G^\ad_{\ol{\QQ}_\lambda}$ and is tame, and the same holds for $\rho_{\pi \leadsto \lambda, v}^\ad$ by \cref{thm:companions} since $\mc{E}_{v, \ol{\QQ}_\pi}^\dag$ has surjective monodromy representation and has unipotent monodromy by \cite[Theorem 6.4.5]{kedlaya:semistable-reduction}. 

The compatibility at special points is given in the following lemma. In fact, we prove the stronger compatibility at special points of the full $G$-local systems, rather than just their adjoint quotients. This stronger compatibility is used in \cite{patrikis:full-compatibility}.

\begin{lem}
Let $s= [x, a] \in \mr{Sh}(\CC)$ be a special point defined in $S_{K_0, s}$ over a number field $E(sK_0)$. Then the canonical local system $\rho_{K_0, s}$ specialized to $sK_0$ defines a compatible system of $G(\QQ_{\ell})$-representations in the following sense: for any choice of Frobenius $\Frob_v$ at a place $v \vert p$ ($p$ not dividing the $N$ from \cref{lem:integral-model}) of $E(sK_0)$, $\rho_{\ell, s}(\Frob_v) \in G(\QQ_{\ell})$ is for all $\ell \neq p$ $G(\QQ_{\ell})$-conjugate to an independent of $\ell$ value $q_{\Frob_v}^{-1} \in G(\QQ)$. When $\ell=p$, as we vary the finite-dimensional representation $\xi$ of $G$, the F-isocrystal $D_{\mr{cris}}(\xi \circ \rho_{p, s}|_{\Gamma_{E(sK_0)_v}})$ on $\Spec(\kappa(v))$ with its linearized crystalline Frobenius $\varphi^{[\kappa(v):\mathbb{F}_p]}$ defines up to conjugacy an element $\gamma_v \in G(\ol{\QQ}_p)$, and $\gamma_v= q_{\Frob_v}^{-1}$ in $[G \git G](\QQ)$.
\end{lem}
\begin{proof}
This is a standard calculation, but we give some details particularly since we have no reference for the crystalline compatibility. Let $\rho_s$ be the specialization of $\rho_{K_0}$ along $s$ and likewise $\rho_{\ell, s}$ for its $G(\QQ_{\ell})$ projection for all $\ell$. By \cite[Lemma 3.2]{patrikis-klevdal:shimura}, for any $\alpha \in \mathbb{A}_{E(sK_0)}^\times$ mapping to $\sigma \in \Gamma^{\mr{ab}}_{E(sK_0)}$, we have
\[
q_{\sigma} a \rho_s(\sigma)= r_x(\alpha)a,
\]
i.e., $\rho_{s}(\sigma)= a^{-1}q_{\sigma}^{-1}r_x(\alpha)a$, for some $q_{\sigma} \in G(\QQ)_+$. For $\sigma \in \Gamma_{E(sK_0)_v}$ and $\ell \neq p$, by $\ell$-projection we obtain $\rho_{\ell, s}(\sigma)=a_{\ell}^{-1} q_{\sigma}^{-1} a_{\ell}$; since the special point is $\ol{\ZZ}[1/N]$-rational, for $\sigma \in I_{E(sK_0)_v}$ we get $q_{\sigma}=1$, and for $\sigma= \Frob_v$ we find $\rho_{\ell, s}(\Frob_v)$ is for all $\ell \neq p$ conjugate (in $G(\QQ_{\ell})$) to the independent of $\ell$ value $q_{\Frob_v}^{-1}$.

Projecting to the $p$-component, for $\sigma \in I_{E(sK_0)_v}$ we thus find $\rho_{p, s}(\sigma)= a_p^{-1} r_x(\alpha) a_p$, where $\alpha \in \mc{O}_v^\times$, and $\rho_{p, s}(\Frob_v)= a_p^{-1} q_{\Frob_v}^{-1} r_x(\varpi_v) a_p$ for the uniformizer $\varpi_v$ mapped to the choice of $\Frob_v$. In particular, up to conjugacy the abelian representation $\rho_{p, s}|_{I_{E(sK_0)_v}}$ is given by the algebraic representation $r_x$, hence is crystalline. In particular, the functor $\Rep(G_{\QQ_p}) \to \FIsoc(\kappa(v))$, 
\[
\xi \mapsto D_{\mr{cris}}(\xi \circ \rho_{p, s}),
\]
composed with the canonical (forgetful) fiber functor to $\mr{Vect}_{E(sK_0)_v}$, gives by the Tannakian formalism an element $\gamma_v \in G(\ol{\QQ}_p)$, well-defined up to inner automorphism (resulting from choosing an isomorphism over some finite extension of $E(sK_0)_v$ between this fiber functor and the canonical fiber functor on $\Rep(G_{\QQ_p})$), such that for all $\xi \in \Rep(G_{\QQ_p})$, $\xi(\gamma_v)$ has the same characteristic polynomial (and even same conjugacy class) as the linearized crystalline Frobenius $\varphi^{[\kappa(v):\mathbb{F}_p]}$ acting on $D_{\mr{cris}}(\xi \circ \rho_{p, s})$. We check this is the characteristic polynomial of $\xi(\inn(a_p)(\rho_{p, s}(\Frob_v)) \cdot r_x(\varpi_v)^{-1})$, i.e. is that of $\xi(q_{\Frob_v}^{-1})$.

There is a finite extension $K/\QQ_p$ such that for any representation $\xi$ of $G_{\QQ_p}$, 
\[
\xi_K \circ \inn(a_p)(\rho_{p, s}|_{\Gamma_{E(sK_0)_v}}) = \bigoplus_i \psi_i
\] 
for crystalline characters $\psi_i \colon \Gamma_{E(sK_0)_v} \to K^\times$. Restricting to $I_{E(sK_0)_v}$, this decomposition induces a decomposition $\xi_K \circ  r_x= \bigoplus_i \chi_i$ for algebraic homomorphisms $\chi_i \colon \Res_{E(sK_0)_v/\QQ_p} \mathbb{G}_m \to \Res_{K/\QQ_p}\mathbb{G}_m$ of $\QQ_p$-tori. A calculation that ultimately rests on Lubin--Tate theory (\cite[Proposition B.4]{conrad:lifting}, to which strictly speaking we should add the corresponding calculation of crystalline periods for unramified representations) shows that $\varphi^{[\kappa(v):\mathbb{F}_p]}$ acts on $D_{\mr{cris}}(\xi_K \circ \inn(a_p)(\rho_{p, s}|_{\Gamma_{E(sK_0)_v}}))$ (canonically a direct sum of free rank-1 $E(sK_0)_v \otimes_{\QQ_p} K$-modules, with $\varphi^{[\kappa(v):\QQ_p]}$ acting $E(sK_0)_v$-linearly) by 
\[
\bigoplus_i \psi_i(\rec_v(\varpi_v)) \chi_i(\varpi_v)^{-1}.
\] 
By the calculation of $\rho_{p, s}$, this implies $\varphi^{[\kappa(v):\mathbb{F}_p]}$ has the same characteristic polynomial as 
\[
\xi_K\left( \inn(a_p)(\rho_{p, s}(\varpi_v)) \cdot r_x(\varpi_v)^{-1} \right)= \xi_K(q_{\Frob_v}^{-1}).\footnote{Our formula is the inverse of that in \cite[Proposition B.4]{conrad:lifting} because we normalize the reciprocity map $\rec_v$ of local class field theory to map uniformizers to geometric Frobenii.}
\]
and thus the same characteristic polynomial as $\xi(q_{\Frob_v}^{-1})$.

Since finite-dimensional representations separate semi-simple conjugacy classes, $\rho_{\ell, s}(\Frob_v)$, $q_{\Frob_v^{-1}}$, and $\gamma_v$ are equal as elements of $[G \git G](\QQ)$.
\end{proof}

\bibliographystyle{amsalpha}
\bibliography{biblio.bib}
\end{document}